\documentclass[psamsfonts,10pt]{amsart}

\usepackage{amssymb,amsfonts}
\usepackage[all,arc]{xy}
\usepackage{hyperref}
\usepackage{mathrsfs}
\usepackage{physics}
\usepackage{tikz}
\usepackage{comment}
\usepackage{caption}
\usepackage{graphicx}
\usepackage{float}

\usepackage[shortlabels]{enumitem}

\def\*#1{\mathbf{#1}}
\newcommand\mc{\mathcal}
\newcommand\ms{\mathscr}

\newcommand\bb{\mathbb}
\newcommand\epsi{\varepsilon}

\newcommand\bs{\setminus}
\newcommand\cl{\overline}
\newcommand\del{\partial}
\newcommand*\dif{\mathop{}\!\mathrm{d}}

\DeclareMathOperator{\vol}{Vol}

\DeclareMathOperator{\diam}{diam}

\DeclareMathOperator{\Isom}{Isom}
\newcommand\td{\widetilde}

\newcommand*\deffont[1]{\textbf{\textit{#1}}}

\newtheorem{thm}[equation]{Theorem}
\newtheorem{cor}[equation]{Corollary}
\newtheorem{prop}[equation]{Proposition}
\newtheorem{lem}[equation]{Lemma}
\newtheorem{conj}[equation]{Conjecture}

\theoremstyle{definition}
\newtheorem{defn}[equation]{Definition}

\newtheorem{exmps}[equation]{Examples}

\theoremstyle{remark}
\newtheorem{rmk}[equation]{Remark}

\numberwithin{equation}{section}

\title[The Eden model on graphs and tessellations of manifolds]{Local behavior of the Eden model on graphs and tessellations of manifolds}

\author[D. (M.) Hua]{Dongming (Merrick) Hua}
\email{dongming@ucsb.edu}
\author[F. Manin]{Fedor Manin}
\email{manin@math.ucsb.edu}
\author[T. Queer]{Tahda Queer}
\email{taqueer@proton.me}
\author[T. Wang]{Tianyi Wang}
\email{tianyiwang@ucsb.edu}

\begin{document}

\begin{abstract}
    The Eden Model in $\mathbb{R}^n$ constructs a blob as follows: initially a single unit hypercube is infected, and each second a hypercube adjacent to the infected ones is selected randomly and infected. Manin, Rold\'{a}n, and Schweinhart investigated the topology of the Eden model in $\mathbb{R}^{n}$ by considering the possible shapes which can appear on the boundary. In particular, they give probabilistic lower bounds on the Betti numbers of the Eden model.

    In this paper, we prove analogous results for the Eden model on any infinite, vertex-transitive, locally finite graph: with high probability as time goes to infinity, every ``possible'' subgraph (with mild conditions on what ``possible'' means) occurs on the boundary of the Eden model at least a number of times proportional to an isoperimetric profile of the graph. Using this, we can extend the results about the topology of the Eden model to non-Euclidean spaces, such as hyperbolic $n$-space and universal covers of certain Riemannian manifolds. 
\end{abstract}

\maketitle

\section{Introduction}
In 1961, Murray Eden proposed a stochastic model to simulate the growth of a bacterial colony or tumor on flat surfaces, known as the Eden growth model, as follows.  Tessellate $\mathbb{R}^n$ with unit cubes. Choose a starting cube to ``infect'' as the origin. Then, at each time step, choose a new cube at random to infect out of those adjacent to the infected cubes.

By the Cox--Durrett shape theorem \cite{cox}, the growth of the Eden model is ball-like (rather than fractal, as one might guess).  That is, over time, the model (after rescaling by $t^{-1/n}$) converges to a specific convex set in $\mathbb R^n$.  However, this leaves a region near the boundary of the convex set, of thickness $\sim \sqrt{t^{-1/n}}$, in which the model can behave in very complicated ways.  Understanding the topology of the Eden model is one way of measuring the amount of complexity in this region, along with e.g.~the area of the boundary studied by Damron, Hanson, and Lam \cite{DHL}.

In previous work by Manin, Rold\'{a}n, and Schweinhart \cite{MRS}, the homology of the Eden model near the boundary in $\bb{R}^n$ was studied in detail. For every $1\leq k\leq n-1$, they proved an asymptotic lower bound for the $k$th Betti number $\beta_k(t)$ of the Eden model at time $t$. More specifically, they showed that there is a constant $C=C(n,k)>0$ such that
\[\beta_k(t)\geq Ct^{\frac{n-1}{n}}\]
with high probability as $t\to\infty$.  This work has recently been strengthened in some cases and extended to other first-passage percolation models by Damron et al.\ \cite{DGLS}.

\begin{figure}
    \centering
    \includegraphics[scale=0.26]{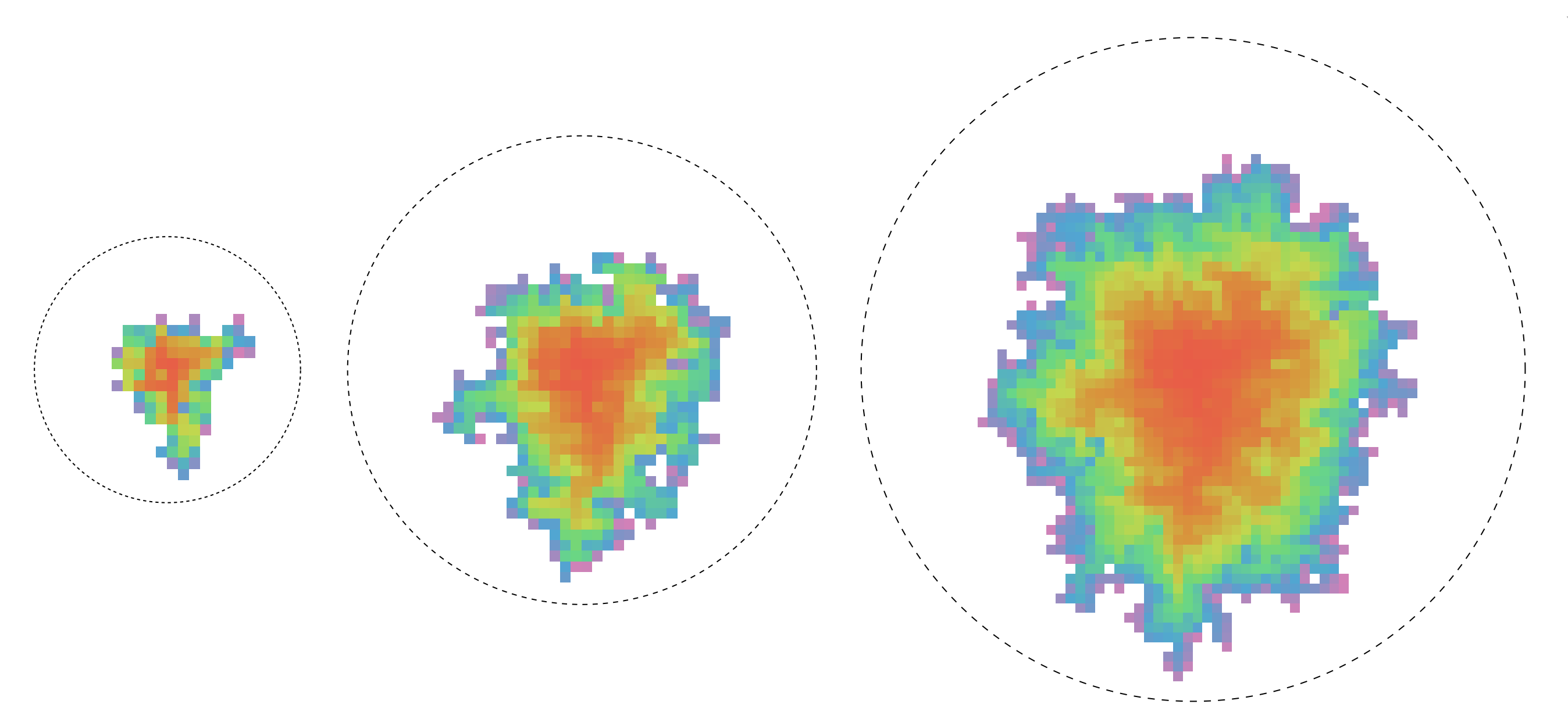}
    \caption{A demonstration of the Eden growth model after 100, 500, and 1500 time steps respectively. This demonstration is written in the Wolfram Language, and is available \href{https://demonstrations.wolfram.com/TumorGrowthModel/}{here}.}
\end{figure}

The definition of the Eden model naturally extends from the cubical lattice in $\mathbb R^n$ to other graphs.  Cayley graphs of groups are a natural setting because like the cubical lattice, which is a Cayley graph of the group $\mathbb Z^n$, they are homogeneous: they look the same from the point of view of every vertex.  Such generalized models have been studied in the context of first-passage percolation on Cayley graphs, mostly in the past decade, although interest in percolation on Cayley graphs more generally dates back earlier \cite{BSch}.  This work has uncovered a contrast based on the growth of the group.  Many of the phenomena first discovered for first-passage percolation in $\mathbb Z^n$ generalize to Cayley graphs of virtually nilpotent groups (equivalently, by Gromov's theorem, groups of polynomial growth), but often not to those of groups of higher growth \cite{BT1,AuGo,Gorski}.  For example, the shape theorem has an analogue for Cayley graphs of virtually nilpotent groups, but in groups of exponential growth no such convergence result can hold \cite{BT1}.  Moreover, first-passage percolation in hyperbolic groups exhibits bi-infinite geodesics, a feature conjectured not to occur in $\mathbb Z^n$ \cite{BT2}.

In this paper, we generalize the results of \cite{MRS} beyond $\mathbb{R}^n$. We first study the Eden model $A(t)$ on an infinite, locally finite, connected, vertex-transitive graph $T$.  This class includes Cayley graphs, but also other examples such as Diestel--Leader graphs which are not Cayley graphs and not even quasi-isometric to a Cayley graph \cite{DL,EFW}.  Subsets of a graph do not have higher-dimensional topology; instead, we prove that certain local patterns occur many times on the perimeter of the Eden model with high probability.  To be precise:
\begin{thm}\label{intro-graphs}
    Let $T$ be an infinite, locally finite, connected, vertex-transitive graph where each vertex has degree $d$.  Let $S$ be any connected subgraph of $T$ which satisfies the following:
    \begin{itemize}
        \item $S$ is contained in the $R$-ball centered at a vertex $x$ of $T$,
        \item $S$ contains the inner boundary of this $R$-ball (all vertices at distance $R$ from $x$).
    \end{itemize}
    Let $A(t)$ denote the Eden model on $T$ at time $t$.  There is a constant $C(R,d)>0$ so that with high probability as $t\to\infty$, there are at least $CF(t)$ disjoint $R$-balls in $T$ whose intersection with $A(t)$ is isomorphic to $S$.
\end{thm}
Here $F(N)$ is the \deffont{isoperimetric profile} of $T$: the minimal boundary of a set of $N$ vertices in $T$.  This is a quasi-isometry invariant (up to a multiplicative constant) of $T$ closely related to the F\o lner function studied e.g.~in \cite{Ersch}.  Notably, the actual boundary of $A(t)$ may grow much faster than the isoperimetric profile.  For example, for amenable groups of exponential growth, such as solvable groups which are not virtually nilpotent, even the perimeter of a ball grows faster.  Specifically, if $T$ is a lattice in the three-dimensional Lie group $\mathrm{SOL}$, then the perimeter of a ball is linear in its volume since both are exponential in the radius, but $F(N) \sim N/\log N$.  It would be interesting to strengthen the bound for such groups.

The proof of Theorem \ref{intro-graphs} uses the reformulation of the Eden model as a first-passage percolation model with passage times distributed according to the exponential distribution.  Our proof uses the memorylessness of the exponential distribution, and another direction for further study is to prove similar results for other distributions, perhaps using the methods of \cite{DGLS}.

We apply Theorem \ref{intro-graphs} to obtain topological results similar to those of \cite{MRS}.  Those results consider the Eden model as defining a random domain in $\mathbb R^n$ consisting of a union of cubical tiles.  Therefore, to generalize them, one looks for other spaces that admit natural tessellations.  The best-known (and, in a precise sense, most symmetric) such space is hyperbolic space: for example, for any pair $(r,s)$, the regular $r$-gon tessellates either the Euclidean plane, the sphere, or the hyperbolic plane with $s$ meeting at every vertex.  Although they are harder to classify, all higher-dimensional hyperbolic spaces also admit infinite families of inequivalent tessellations by convex polyhedra, making this setting considerably richer than the Euclidean one, where there are only finitely many types of tessellation.  The Eden model on such a tessellation is defined similarly to the usual Euclidean Eden model: one starts with a single infected cell, and at each time step, one chooses a random adjacent tile to infect.

\begin{thm} \label{hyperbolic intro}
  Consider a regular tessellation of the hyperbolic $n$-space $\bb H^n$ by compact, convex tiles.  (Here, \deffont{regular} means that the symmetry group of the tessellation acts transitively on the cells.)  Let $A(t)$ be the Eden model at time $t$ on this tessellation, and let $\beta_k(t)$ denote the $k$th Betti number of $A(t)$.  Then for every $1 \leq k \leq n-1$, there are constants $C>c>0$ depending on $k$ and the tessellation such that
  \[ct \leq \beta_k(t) \leq Ct\]
  with high probability as $t \to \infty$.
\end{thm}

Here the growth is linear since hyperbolic spaces have a linear isoperimetric inequality: the perimeter of every set of tiles is linear in its volume.

We also discuss a generalization that encompasses both the hyperbolic and Euclidean settings.  We can think of the unit cubes of the Eden model on $\mathbb R^n$ as fundamental domains of the universal cover of $T^n$.  Similarly, the tiles of many hyperbolic tessellations are fundamental domains of the universal cover of a hyperbolic manifold.  For any compact manifold $M$, we define a similar Eden model which builds random subsets---unions of adjacent fundamental domains---of the universal cover of $M$.  We prove a result similar to Theorem \ref{hyperbolic intro} for the class of manifolds with non-positive sectional curvature.  This class includes the Euclidean (curvature zero) and hyperbolic (constant curvature $-1$) cases, but also a large class of tessellations of manifolds which do not admit any constant curvature metric, such as symmetric spaces of noncompact type, as well as countless examples with less local symmetry.
\begin{thm} \label{NPC intro}
  Let $M$ be a closed Riemannian manifold of non-positive sectional curvature, and let $A(t)$ be the Eden model at time $t$ on the universal cover $\widetilde M$.  Let $\beta_k(t)$ denote the $k$th Betti number of $A(t)$ with coefficients in $\mathbb Z/2\mathbb Z$, and let $F(N)$ denote the isoperimetric profile of the Cayley graph of $\pi_1(M)$.\footnote{We show in Appendix \ref{appendix} that the isoperimetric profiles of two Cayley graphs of the same group differ only by a multiplicative constant.}  Then for every $1 \leq k \leq \dim M-1$ there is a constant $C=C(M,k)>0$ such that
  \[\beta_k(t) \geq CF(t)\]
  with high probability as $t \to \infty$.

  In particular, if $M$ is negatively curved, then $\beta_k(t)$ grows linearly in $t$ (since growth of the lower bound matches that of the trivial upper bound).
\end{thm}

In this case, unlike in Theorem \ref{hyperbolic intro}, we don't give a matching upper bound.

Moreover, if the fundamental domains we choose (or their intersections) themselves have topology, then the lower bound in Theorem \ref{NPC intro} may in fact be trivial, in the sense that it holds for every subset of $T$ consisting of $t$ fundamental domains.  One might even worry that this is true for all possible choices of fundamental domain, and therefore our results are completely uninteresting.  To forestall this objection, we show that if $M$ has constant curvature, one can choose a tessellation (by Voronoi cells) for which any nonempty intersection of cells is contractible.  It would be interesting to construct tessellations with this property in a more general setting, perhaps even for any aspherical manifold:

\begin{conj}
  If $M$ is an aspherical manifold, there is a compact set $K$ in its universal cover $\widetilde M$ such that translates of $K$ under the deck action of $\pi_1(M)$ cover $\widetilde M$ and any nonempty intersection of these translates is contractible.
 \end{conj}
If this condition is satisfied, then the union of a set of translates of $K$ is homotopy equivalent to the nerve of the covering by tiles; in other words, the topology is determined by the lattice of intersections and not any local information about $K$.

\section{Preliminaries}
In this section we collect some background knowledge required to read the rest of the paper. 
\subsection*{Graph theory}
One of the main settings in which we will study the Eden model is graphs. More specifically, we will consider the Eden model on infinite, connected, locally finite, vertex-transitive graphs. Here, \deffont{locally finite} means that each vertex in the graph has finite degree, and \deffont{vertex-transitive} means that for any vertices $v, w$ in our graph, there is a graph automorphism sending $v$ to $w$. Informally, vertex-transitive graphs ``look the same'' at each vertex. One large class of infinite, connected, locally finite, vertex-transitive graphs is given by the Cayley graph of any infinite, finitely-generated group with respect to a finite generating set.

The Eden model on an infinite, connected, locally finite, vertex-transitive graph is defined as follows: at time $t = 0$, an arbitrary vertex of the graph is infected (it does not matter which one, since the graph is vertex-transitive). Then, at each integer time, a vertex adjacent to at least one of the previously infected vertices is randomly chosen (with equal probability) and infected. For a given graph, we will use $A(t)$ to denote the Eden model on the graph at time $t$.

Note that every infinite, connected, locally finite, vertex-transitive graph has a well-defined degree $d$, such that each vertex in the graph has degree $d$. This follows from the fact that the graph is locally finite and vertex-transitive, hence each vertex has the same finite degree. This gives us an analogue of the dimension of Euclidean space for graphs. 

Furthermore, we can consider each infinite, connected, locally finite, vertex-transitive graph as a metric space, by endowing it with the \deffont{graph metric}: the distance between any two vertices is the length of the shortest edge path between them. Since our graph is always connected, such a path always exists. The closed balls in the graph metric will be important to our results.

\subsection*{Geometric group theory}
The language and techniques of geometric group theory occur frequently in this paper, so we briefly collect the most important ones here. The proofs of the results in this section can be found in standard books on geometric group theory, such as \cite{officehours}, \cite{DruKap}, or \cite{claraloh}.

\begin{defn} \label{CayleyGraph} Given a group $G$ with a finite generating set $S$, the \deffont{Cayley graph} $\operatorname{Cay}(G, S)$ has vertex set $G$ and an edge between $g$ and $h$ iff $gh^{-1} \in S \cup S^{-1}$. 
\end{defn}

We quickly prove a recognition theorem for Cayley graphs:
\begin{lem}\label{recognizeCayley}
Let $T$ be a connected, locally-finite graph with $|T| > 1$, and suppose $G$ is a group acting on $T$ by graph isomorphisms (for each $g$, $x \mapsto gx$ is a graph isomorphism of $T$). If the action of $G$ on $T$ is free and transitive, then $T$ is a Cayley graph for $G$.
\end{lem}
\begin{proof}
Pick a vertex $x \in T$. Since $G$ acts transitively and freely on $T$, each vertex of $T$ can be written uniquely as $gx$ for some $g \in G$. 

Let 
\[S = \{g \in G \colon x, gx\text{ are adjacent}\}.\]
Since $T$ is connected and locally-finite, $S$ is nonempty and finite. We claim that $S$ generates $G$. 

First, note that since $x \mapsto g^{-1}x$ is a graph isomorphism, $gx$ and $hx$ are adjacent if and only if $x$ and $g^{-1}hx$ are adjacent.  In other words, $gx, hx \in T$ are adjacent if and only if $g^{-1}h$ is in $S$. (Note that this also shows $S = S^{-1}$.)

Let $h \in G$ be an arbitrary nonidentity element. Since $T$ is connected, there exists a path in $T$ from $x$ to $hx$. In other words, there exist vertices $g_{1}x, g_{2}x, \ldots, g_{n}x \in T$ such that $x$ and $g_{1}x$ are adjacent, $g_{i}x, g_{i+1}$ are adjacent for all $i \in \{1, \ldots, n-1\}$, and $g_{n}x$ is adjacent to $hx.$ By the previous paragraph, we have 
$$g_{1}, g_{i}^{-1}g_{i+1}, g_{n}^{-1}h \in S$$ for all $i \in \{1, \ldots, n-1\}$. We have 
$$h = g_{1}(g_{1}^{-1}g_{2})(g_{2}^{-1}g_{3})\cdots(g_{n}^{-1}h),$$ which proves any element in $G$ can be written as the product of elements in $S$, hence $S$ generates $G$.

So, we can define the Cayley graph $\text{Cay}(G, S)$. Let $f \colon \text{Cay}(G, S) \to T$ be the map sending $g \in \text{Cay}(G, S)$ to $gx$. This is a bijection, since each vertex of $T$ can be written uniquely as $gx$ for some $g \in G$. 

Furthermore, since $g, h \in \text{Cay}(G, S)$ are adjacent iff $g^{-1}h \in S \cup S^{-1} = S$, and $gx, hx \in T$ are adjacent iff $g^{-1}h \in S$, we conclude that $g, h$ are adjacent iff $f(g) = gx$ and $f(h) = hx$ are adajcent, hence $f$ is a graph isomorphism.
\end{proof}

Equipping $\text{Cay}(G, S)$ with the graph metric, we obtain a metric space associated with $G$.  While the Cayley graphs of a group $G$ with respect to different generating sets $S, S'$ are usually not isometric, they are metrically equivalent in a weaker sense:
\begin{defn}\label{Quasi-Isometry}
Let $X$, $Y$ be metric spaces and let $C \geq 1$ and $K \geq 0$ be constants. A function $f \colon X \to Y$ is a \deffont{$(C, K)$-quasi-isometry} if
\begin{itemize}
    \item (Quasi-isometric embedding) For any $x, x' \in X$, 
    \[\frac{1}{C}d_{Y}(f(x), f(x')) - K \leq d_{X}(x, x') \leq Cd_{Y}(f(x), f(x'))+K.\]
    \item (Coarse surjectivity) For any $y \in Y$, there exists some $x \in X$ such that $d_{Y}(f(x), y) \leq K$.
\end{itemize}
If for some $C$ and $K$ there is a $(C,K)$-quasi-isometry $f \colon X \to Y$, then $X$ and $Y$ are \deffont{quasi-isometric}. This is an equivalence relation.
\end{defn}

\begin{prop}\label{different generating sets quasi-isometric}
Given a group $G$ and two finite generating sets $S, S'$, the map $f \colon \operatorname{Cay}(G, S) \to \operatorname{Cay}(G, S')$ sending $g$ to $g$ is a quasi-isometry.
\end{prop}
As a consequence of this proposition, a finitely generated group $G$ is well-defined as a metric space up to quasi-isometry, and so we can talk about groups being quasi-isometric to one another or to other spaces.

\begin{defn}\label{proper metric space}
A metric space $X$ is \deffont{proper} if closed balls in $X$ are compact.
\end{defn}

\begin{defn}
A metric space $X$ is a \deffont{geodesic space} if for any $x, y \in X$, there exist $a, b \in \bb R$ with $a \leq b$ and an isometric embedding $\gamma \colon [a, b] \to X$ such that $\gamma(a) = x$ and $\gamma(b) = y$.  Such an isometric embedding is called a \deffont{geodesic}.
\end{defn}
Note that a geodesic in this metric sense is a shortest path between any of its points.  This does not always coincide with the variational definition of a geodesic in Riemannian geometry.  However, in a complete Riemannian manifold, a metric geodesic is always a Riemannian geodesic; conversely, there is always a shortest Riemannian geodesic between two points and this shortest path is always a metric geodesic.  In particular, every complete Riemannian manifold is a geodesic space.

\begin{defn}\label{geometric group action}
Let $G$ be a group and $X$ a metric space. A group action $G \curvearrowright X$ is \deffont{geometric} if it satisfies the following properties:
\begin{itemize}
\item It \deffont{acts by isometries}: for each $g \in G$, the map $X \to X$ sending $x$ to $g \cdot x$ is an isometry.
\item It is \deffont{cocompact}: the quotient $X/G$ is a compact space.
\item It is \deffont{properly discontinuous}: for any compact $K \subseteq X$, there are only finitely many $g \in G$ such that $K \cap g \cdot K \neq \emptyset$.
\end{itemize}
\end{defn}

A central result in geometric group theory is the Milnor--Schwarz Lemma, which states the following:
\begin{thm}[Milnor--Schwarz]\label{milnorschwarz}
If $G$ acts geometrically on a proper geodesic space $X$, then $G$ is finitely generated, and for any finite generating set $S$ of $G$ and any $x \in X$ the map $f_{x}$ sending $g \in G$ to $g \cdot x \in X$ is a quasi-isometry from $\text{Cay}(G, S)$ to $X$.
\end{thm}

One important source of geometric group actions come from fundamental groups of manifolds. Let $M$ be a compact, connected Riemannian manifold, and equip the universal cover $\widetilde{M}$ with the unique Riemannian metric such that the covering map $p \colon \widetilde{M} \to M$ is a local isometry. Then, the action of $\pi_{1}(M)$ on $\widetilde{M}$ by deck transformations is geometric \cite[Corollary 5.4.10]{claraloh}.

Another source of geometric group actions come from lattices of isometries of hyperbolic $n$-space, which will be discussed in more detail in the next subsection.

\subsection*{Hyperbolic geometry}
Hyperbolic $n$-space $\bb{H}^{n}$ is the unique (up to isometry) simply-connected $n$-dimensional Riemannian manifold of constant sectional curvature $-1$.  There are multiple models of $n$-dimensional hyperbolic space, but the one we use in this paper is the \deffont{Poincar\'{e} disc model}: we view $\bb{H}^{n}$ as the open unit ball in $\bb{R}^{n}$ with Riemannian metric
\[\dif s^2=\frac{4(\dif x_1^2+\cdots+\dif x_n^2)}{(1-(x_1^2+\cdots+x_n^2)^2)^2}.\]
Much of the material in this subsection can be found in \cite[Chapter 4]{DruKap}.

Hyperbolic spaces are particularly symmetric: for every $x,y \in \bb H^n$ and every orthogonal transformation $L:T_x\bb H^n \to T_y\bb H^n$, there is an isometry $f:\bb H^n \to \bb H^n$ such that $f(x)=y$ and $Df_x=L$.  Since isometries take geodesics to geodesics, each such isometry is unique.  Therefore, there is a bundle structure
\[O(n) \to \Isom(\bb H^n) \xrightarrow{p} \bb H^n,\]
where the map $p$ sends an isometry $f$ to $f(x)$, for an arbitrary basepoint $x$.

In particular, the group of isometries fixing $x$ is isomorphic to $O(n)$.  Moreover, these subgroups are conjugate for every choice of $x$.  Therefore, to understand them, it is enough to consider the isometries fixing the origin in the Poincar\'e disk model.  These are exactly the orthogonal transformations of $\bb R^n$, since the Poincar\'e disk metric is rotationally invariant.  The fixed set of each such transformation is a linear subspace of $\bb R^n$; its intersection with the Poincar\'e disk is isometric to a lower-dimensional hyperbolic space inside $\bb H^n$.  We state this as a lemma:
\begin{lem} \label{hyp fixed set}
    For every nontrivial isometry of hyperbolic space, the fixed set, if nonempty, is a hyperbolic subspace of lower dimension.
\end{lem}

As a Lie group, the isometry group is isomorphic to $O^+(n,1)$, an index 2 subgroup of the group of linear isometries of the quadratic form
\[x_{1}^{2} + \cdots + x_{n}^{2} - x_{n+1}^{2}.\]

\begin{defn} \label{def:lattice}
A \deffont{lattice} $\Gamma$ in $\operatorname{Isom}(\bb{H}^{n})$ is a discrete subgroup $\Gamma < \operatorname{Isom}(\bb{H}^{n})$, where $\text{Isom}(\bb{H}^{n})$ is given the compact-open topology. We say $\Gamma$ is a \deffont{cocompact} or \deffont{uniform lattice} if the quotient $\operatorname{Isom}(\bb{H}^{n})/\Gamma$ is compact.
\end{defn}
If $\Gamma$ is a cocompact lattice of isometries of $\bb{H}^{n}$, then the action of $\Gamma$ on $\bb{H}^{n}$ is geometric.  This is a special case of a more general fact: if $M$ is a homogeneous Riemannian manifold, and $\Gamma$ is a cocompact lattice in $\operatorname{Isom}(M)$, then $\Gamma$ acts geometrically on $M$ \cite[p. 142]{DruKap}.

Cocompact lattices in $\bb H^n$ are abundant, with infinitely many distinct types in every dimension $n \geq 2$, as demonstrated by work of Borel, Mostow, and Gromov and Piatetski-Shapiro among others; see \cite[Chapter 12]{DruKap} for an overview.

In this paper we use such lattices to obtain tilings of $\bb{H}^{n}$ by convex sets called Voronoi cells, on which we can then define a version of the Eden model.  The full construction is given in Section \ref{The Eden model on hyperbolic tessellations}.



\subsection*{Homology and cohomology}
We use standard tools from homology theory which can be found in any algebraic topology textbook such as \cite{hatcher}.  Our results hold for homology both with integer coefficients and with coefficients in any field.

\subsection*{First-passage percolation and the Eden model}
First-passage percolation (FPP) is a well-studied family of stochastic growth models; see \cite{ADH} for an extensive survey of FPP on $\mathbb Z^n$, and \cite{BT1,BT2} for some results on other Cayley graphs.  In our proof in the following sections, we will make use of a kind of equivalence, first noticed by Richardson \cite{richard}, between the Eden model and a specific FPP model: site FPP with exponentially distributed passage times. While our results can be interpreted equally as applying to the Eden model and to this FPP model, the proof requires working with the FPP model.

We first define site FPP on a graph $T=(V,E)$. Every vertex of $T$ is called a \deffont{site}. Every site $p$ is assigned an independent and identically distributed (i.i.d.) number $\rho_p$ called the \deffont{passage time}. The passage time determines the following: if the first site adjacent to $p$ gets infected at time $t$, then at time $t+\rho_p$, the site $p$ is infected.  A graph with passage times assigned to each vertex, together with an initial vertex which is infected at time $0$, determines the infected region at every time $t>0$.  We usually call this region the \deffont{FPP ball} at time $t$ to emphasize its similarity to a ball around the initial vertex with respect to a random metric.  A probability distribution on $\rho_p$ determines an \deffont{FPP model}, which is a distribution on the set of functions $V \to [0,\infty)$ determining the passage times of each site.

We choose our i.i.d.\ passage times from an exponential distribution with mean $1$.  The exponential distribution is \deffont{memoryless}, that is,
\[\bb{P}(X>t+s|X>s)=\bb{P}(X>t).\]
In particular, if at time $t$ a site $p$ is located at distance $1$ from the infected region, the additional time required to infect $p$ is again distributed exponentially with mean $1$, and does not depend on the passage times of any other site.

Moreover, the probability of two sites adjacent to $A$ having the same passage time is zero (that is, two sites are never filled in simultaneously).  Therefore, for a creature that can only tell the order in which events happen, but not the duration between them, the site FPP model with exponentially distributed passage times will be statistically indistinguishable from the Eden growth model. In this sense, an Eden model is just an FPP model with time rescaled. This correspondence was first observed by Richardson \cite{richard}. Therefore, in the remainder of this paper, we will consider the FPP model with exponentially distributed ($\mu=1$) passage times, as it is easier to work with than the Eden model itself.

One possible issue is this.  We will prove several results that state that an event happens \deffont{with high probability}, that is, that the probability approaches $1$ as time increases.  We will prove this with respect to FPP time, but would like to state it with respect to Eden time.  In Section \ref{eden vs fpp}, we justify this discrepancy by showing that in any infinite, connected, locally finite, vertex-transitive graph, as Eden time grows, FPP time also grows without bound with high probability.  (In a Euclidean grid, this follows from the shape theorem, but in general, the relationship may be considerably looser.)

\section{The Eden model on graphs}

While the homology of the Eden model on graphs is not very interesting, we can instead generalize one of the main results shown in \cite{MRS}, a consequence of which is the lower bound on homology. In \cite{MRS}, it is proved that for a given “nice” subset of $\mathbb{R}^{n}$ built up from unit hypercubes, with high probability we can find at least $Ct^{\frac{n-1}{n}}$ copies of this subset on the boundary of the Eden model in $\mathbb{R}^{n}$. In this section, we will show an analogous result for graphs. 

First, we introduce some definitions. 

\begin{defn}[Boundary]
Let $A = (V_{A}, E_{A})$ be a subgraph of a graph $T=(V,E)$. Then the \deffont{boundary} $\del A$ of $A$ is
\[\del A = \qty{w \in V \setminus V_{A}: \exists v \in A, (v, w) \in E.}\]
\end{defn}

This definition, although it is convenient for our results, is non-standard: the boundary $\del_{\text{edge}} A$ of $A$ is usually defined to be the set of edges connecting vertices in $A$ to vertices outside $A$.  However, it is not hard to see that
\[|\del A| \leq |\del_{\text{edge}}(A)| \leq \deg T|\del A|.\]
That is, the two notions of perimeter (size of the boundary) differ by at most a multiplicative constant, which means that many ideas are unaffected by substituting one for the other, such as the definition of non-amenability below.

\begin{defn}
Let $T = (V, E)$ be a connected, locally finite graph with the graph metric $d$. Given a vertex $x \in T$ and a subgraph $A$ of $T$, we let 
\[d(A, x) = \min\{d(x, v) \colon v \in A\}.\]
\[B_{R}(A) = \{v \in T \colon d(A, v) \leq R\}\]
\[D_{R}(A) = \{v \in T \colon d(A, v) = R\}.\]
When the subscript is too long, we will sometimes use $B(A, R)$ and $D(A, R)$ to denote the same sets. Note that $\partial A = D_{1}(A)$, and that in general we have $\partial B_{R}(A) = D_{R+1}(A)$.
\end{defn}

The isoperimetric inequality in $\mathbb{R}^{n}$, states that given a fixed amount of volume, the shape with the least perimeter that encloses at least that much volume is a sphere, hence any shape has perimeter greater than or equal to that of a sphere of the same volume.  The term $t^{\frac{n-1}{n}}$ in the lower bound on the homology of the Euclidean Eden model derives from this isoperimetric inequality. Our generalization to graphs involves defining a corresponding notion of isoperimetry for graphs. 

\begin{defn}
    The \deffont{F\o lner isoperimetric profile} $F:\bb{N}\to\bb{N}$ of a graph $T=(V,E)$ is defined by
    \[F(n)=\min\qty{|\del A|: |A|\geq n}\]
    where $A$ is any subgraph of $T$ with size at least $n$.
\end{defn}

This definition is nonstandard in two ways. Firstly, it is not defined in terms of the more standard edge boundary. However, one can verify that if $F_{\text{edge}}$ is the F\o lner isoperimetric profile defined in terms of the edge boundary, then by the fact that 
\[|\del A| \leq |\del_{\text{edge}}(A)| \leq \deg T|\del A|\]
we also have
$$F(n) \leq F_{\text{edge}}(n) \leq \deg T \cdot F(n).$$ The first inequality follows from choosing $A$ such that $|\del_{\text{edge}}(A)| = F_{\text{edge}}(n)$, and the second follows from choosing $A$ such that $|\del A| = F(n)$. Since ultimately we are only concerned with the isoperimetric profile up to a multiplicative constant, this means the two are interchangeable for our purposes.

Secondly, $F$ is defined ranging over all subgraphs of size at least $n$, whereas it is more common to define $F$ ranging over all subgraphs of size exactly $n$. The reason we choose to define it this way is so that $F$ becomes a nondecreasing function, which is necessary for certain arguments to go through.

\begin{prop}\label{sublinearity}
    Let $T$ be an infinite vertex-transitive graph. Then
    \[F(m+n)\leq F(m)+F(n).\]
    In particular, if $k$ and $n$ are positive integers, we have $F(kn)\leq kF(n)$. 
\end{prop}

\begin{proof}
    From the definition of $F$, there exist subgraphs $A_m,A_n$ of $T$ with $|A_{m}| = m$, $|A_{n}| = n$ which attain the minimum for $F$, i.e. $|\del A_m|=F(m)$ and $|\del A_n|=F(n)$. Furthermore, since $T$ is infinite and vertex-transitive, we can choose $A_{m}, A_{n} \subseteq T$ so that $d(A_{m}, A_{n}) \geq 3$. In particular, $A_{m} \cap A_{n} = \emptyset$ and $|\partial A_{m}| \cap |\partial A_{n}| = \emptyset$.
    
    Now take $A=A_m\cup A_n$. Then $|\del A|=|\del A_m|+|\del A_n|=F(m)+F(n)$, and $|A|=m+n$. Therefore, by the definition of $F$,
    \[F(m+n)\leq |\del A|=F(m)+F(n).\]
    The second statement follows by induction on $k$.
\end{proof}

\begin{exmps} \label{IP examples}
\begin{enumerate}[(a)]
\item The classical isoperimetric inequality implies that the isoperimetric profile of the Cayley graph of $\bb Z^d$ is $Cn^{\frac{d-1}{d}}$ for some constant $C$.
\item In any regular tree, the isoperimetric profile is linear in $n$.
\item In general, a graph whose isoperimetric profile is bounded below by a linear function is called \deffont{non-amenable}.  This is one of many equivalent definitions of non-amenability, due to F\o lner \cite{folner}.  Let $\td M$ be a complete manifold of negative sectional curvature and let $\Gamma$ be a group acting on $\td M$ geometrically.  Then $\td M$ is a hyperbolic metric space \cite[Example 7.2.3]{claraloh}, and $\Gamma$ is a Gromov hyperbolic group.  In particular, this means that $\Gamma$ is non-amenable \cite[Corollary 9.1.11]{claraloh} and therefore the isoperimetric profile of $\Gamma$ is bounded below by a linear function.  This is related to the fact that $\td M$ also has a linear isoperimetric profile (in the Riemannian sense).
\end{enumerate}
\end{exmps}

    

The main result of this section is Theorem \ref{intro-graphs}, which we restate here:

\begin{thm}\label{generalized theorem 8}
    Let $T$ be an infinite, locally finite, connected, vertex-transitive graph where each vertex has degree $d$.  Let $S$ be any connected subgraph of $T$ which satisfies the following:
    \begin{itemize}
        \item $S$ is contained in an $R$-ball $B_{R}(x) = \{v \in T \colon d(v, x) \leq R\}$ whose center $x$ is a vertex of $T$,
        \item $S$ contains $D_{R}(x) \subset B_R(x)$.
    \end{itemize}
    Let $A(t)$ denote the Eden model on $T$ at time $t$. There is a constant $C(R,d)>0$ so that with high probability as $t\to\infty$, there are at least $CF(|A(t)|)$ disjoint $R$-balls in $T$ whose intersection with $A(t)$ is isomorphic to $S$.
\end{thm}
Note that since $T$ is connected, we can assume $d \geq 2$.

To prove Theorem \ref{generalized theorem 8}, we proceed as follows.  We first show that at any time $t$, we can put $\Omega(F(|A(t)|))$ disjoint $R$-balls adjacent to $A(t)$, in the sense that these balls touch $A(t)$ at the boundary, but they do not intersect with $A(t)$. This is proven in Lemma \ref{enough ball lemma}. In particular, we can find $\Omega(F(|A(t-2)|))$ such balls adjacent to $A(t-2)$. Lemma \ref{generalized lem 9} shows that this implies we can find $\Omega(F(|A(t)|))$ many such balls adjacent to $A(t-2)$. Finally, Lemma \ref{enough fraction ball lemma} shows that at time $t$, a positive fraction of these balls will intersect with $A(t)$ in a translate of $S$.

\begin{rmk} \label{rmk:strengthening}
    Further in the paper, we actually use a slight strengthening of Theorem \ref{generalized theorem 8}.  Fix a base vertex $v_0$ of $T$ and suppose that for every $v \in T$ we have fixed a graph automorphism $\phi_v:T \to T$ that takes $v_0$ to $v$.  Then we can choose $CF(|A(t)|)$ disjoint balls such that whenever $v$ is the center of one of them, then $B_R(v) \cap A(t)=\phi_v(S)$.  One can see that this is true by tracing through the proof of the theorem.
\end{rmk}

We first give a rough estimate of the volume of $R$-balls in our graph:

\begin{lem}\label{size of balls}
    Let $T$ be an infinite vertex-transitive graph of degree $d \geq 2$, and let $B_R(x)$ denote the ball of radius $R$ centered at $x\in T$. Then $|B_R(x)|\leq d^{R+1}$.
\end{lem}

\begin{proof}
    Recall that $D_{R}(x)$ denotes the set of points in $T$ at distance $R$ from $x$. Note that for each $R \geq 1$, every point in $D_{R}(x)$ is distance one from some point in $D_{R-1}(x)$, and therefore
    \[D_{R}(x) \subseteq \bigcup_{y \in D_{R-1}(x)} B_{1}(y).\]
    Since each vertex has degree $d$, we have $|B_{1}(y)| = d$ for all $y$, and hence
    \[|D_{R}(x)| \leq \left|\bigcup_{y \in D_{R-1}(x)} B_{1}(y) \right| \leq |D_{R-1}(x)| \cdot d.\]
    
    Since $|D_{0}(x)| = 1$, by induction we have $|D_{R}(x)| \leq d^{R}$ for all $R$. It follows that 
    \[|B_R(x)| = \left|\bigcup_{i = 0}^{R} D_i(x) \right| \leq 1+d+d^2+\cdots+d^R = \frac{d^{R+1}-1}{d-1} \leq d^{R+1}. \qedhere\]
\end{proof}

\begin{lem}\label{enough ball lemma}
    There is a constant $C_1(R,d)>0$ such that we can find $C_1F(|A(t)|)$ disjoint balls of radius $R$ which are adjacent to but do not intersect $A(t)$.
\end{lem}

\begin{proof}
    We use $A$ as a shorthand notation for $A(t)$. As before, $D_{R}(A)$ denotes the set of points in $T$ whose distance from $A$ is exactly $R$.

    As noted above, we have $D_{R+1}(A) = \del B_{R}(A)$, and we have $|A| \leq |B_{R}(A)|$ (because $A \subseteq B_{R}(A)$. So, by definition of the F\o lner isoperimetric profile,
    \[|D_{R+1}(A)|=|\del B_{R}(A)|\geq F\left(|B_{R}(A)|\right) \geq F\left(|A|\right).\]
    
    Now consider the set of all possible centers of $R$-balls which are adjacent to but do not intersect $A(t)$.  These are exactly the points of $D_{R+1}(A)$.  We now pick the centers of our balls by a greedy procedure. First, pick any $x_1\in D_{R+1}(A)$ as the center of our first $R$-ball.  Then any vertex $v$ with $d(v,x_1)\leq 2R$ in the graph metric is no longer an available candidate to be the center of another such $R$-ball (since we need the $R$-balls to be disjoint).  At worst, we have removed $\vol(B_{2R})\leq d^{2R+1}$ points of $D_{R+1}(A)$ from the candidate pool. Repeatedly applying this argument, we see that we can put at least
   \[\left\lceil \frac{|D_{R+1}(A)|}{d^{2R+1}}\right\rceil  \geq  \left\lceil \frac{F(|A|)}{d^{2R+1}}\right\rceil \geq d^{-(2R+1)}F(|A|)\]
    disjoint balls of radius $R$ which are adjacent to but do not intersect $A(t)$.  This completes the proof with $C_1=d^{-(2R+1)}$.
\end{proof}

\begin{lem}\label{generalized lem 9}
Let $G$ (we use $G$ to denote a graph here since we use $T$ to denote a tree in the proof) be an infinite, connected, locally finite, vertex-transitive graph of degree $d$, and let $A(t)$ be the ball at time $t$ in the mean $1$ exponentially distributed FPP model on $G$.  There is a constant $C_2>0$ depending on $d$ such that 
\[F(|A(t)|)\leq C_2F(|A(t-2)|)\]
with high probability as $t\to\infty$.
\end{lem}

\begin{proof}
    We will show that there is a constant $C_2$ such that
    \begin{equation}\label{A(t)<<A(t-2)}
        |A(t)|\leq C_2|A(t-2)|
    \end{equation} 
    with high probability as $t\to\infty$.  Without loss of generality, we can assume $C_2$ is an integer (otherwise take its ceiling). Since $F$ is nondecreasing by definition and subadditive by Proposition \ref{sublinearity}, \eqref{A(t)<<A(t-2)} implies that
    \[F(|A(t)|) \leq F(C_2|A(t-2)|) \leq C_2F(|A(t-2)|),\]
    completing the proof.
    
    To prove \eqref{A(t)<<A(t-2)}, we will show that there is a constant $\epsi>0$ such that with high probability as $t\to\infty$ we have
    \begin{equation}\label{A(t+epsi)<<A(t)+del A(t)}
        |A(t+\epsi)|\leq |A(t)|+|\del A(t)|.
    \end{equation}
    In particular we have $|A(t+\epsi)|\leq (d+1)|A(t)|$. Then taking $C_2=(d+1)^{\lceil 2/\epsi\rceil}$ does the trick.
    
    To prove \eqref{A(t+epsi)<<A(t)+del A(t)}, we first consider a rooted tree $T$ of degree $d-1$ whose nodes are equipped with i.i.d., exponentially distributed passage times with mean $1$. Choose $\epsi>0$ such that for any node $p'\in T$, we have
    \[\bb{P}(\rho_{p'}<\epsi)=\frac{1}{d+1}.\]
    Let $T(\epsi) \subseteq T$ be the maximal subtree of $T$ containing the root whose nodes all have passage time less than $\epsi$.  Then the expected size $E$ of $T(\epsi)$ satisfies the recurrence relation
    \[E=\frac{1}{d+1}(1+(d-1)E).\]
    Solving this we get $E=\frac{1}{2}$. 
    
    Now we start to prove \eqref{A(t+epsi)<<A(t)+del A(t)}. Fix $t>0$.  We give a coupling between the distribution of passage times on a collection of $(d-1)$-ary rooted trees and the distribution of passage times for sites in $G$ outside $A(t)$.  We describe this coupling by assigning passage times to sites in $G$ using the passage times on the collection of trees.  With every $x\in \del A(t)$ we associate a tree $T_x$ and a graph homomorphism $f_x:T_x\to G$.  This map sends the root $o_x$ of the tree to $x$, and it sends the $d-1$ vertices adjacent to $o_x$ to $d-1$ different vertices adjacent to $x$, avoiding the one vertex contained in $A(t)$; and from there the map is defined inductively so that the $d$ vertices adjacent to a given $v$ are mapped to all $d$ vertices adjacent to $f_x(v)$.  Thus we obtain a map 
    \[f:\coprod\nolimits_{x\in \del A(t)}T_x\to G\]
    such that the root $o_x$ of $T_x$ is mapped via $f$ to $x \in \del A(t)$, and every path in $G \setminus A(t)$ connecting a vertex in $\del A(t)$ to another vertex in $G \setminus A(t)$ corresponds to a unique rooted path in $\coprod_{x} T_x$.
    
    See Figure \ref{coupling tree} for an intuitive picture for this construction.

\begin{figure}
\centering

\tikzset{every picture/.style={line width=0.75pt}} 

\begin{tikzpicture}[x=0.75pt,y=0.75pt,yscale=-1,xscale=1]

\draw  [color={rgb, 255:red, 0; green, 0; blue, 0 }  ,draw opacity=1 ][fill={rgb, 255:red, 155; green, 155; blue, 155 }  ,fill opacity=0.2 ] (274.93,144.04) .. controls (290.68,135.2) and (320.43,151.35) .. (333.68,170.7) .. controls (346.93,190.04) and (353.93,209.04) .. (333.68,230.7) .. controls (313.43,252.35) and (253.53,245.47) .. (243.68,230.7) .. controls (233.83,215.93) and (232.08,195.63) .. (238.85,179.88) .. controls (245.61,164.14) and (259.19,152.88) .. (274.93,144.04) -- cycle ;
\draw [color={rgb, 255:red, 0; green, 0; blue, 0 }  ,draw opacity=1 ][line width=2.25]    (302.4,145.87) -- (294,130.2) ;
\draw [color={rgb, 255:red, 74; green, 74; blue, 74 }  ,draw opacity=1 ]   (317.4,139.87) -- (302.4,145.87) ;
\draw [color={rgb, 255:red, 0; green, 0; blue, 0 }  ,draw opacity=1 ][line width=2.25]    (284.81,125.28) -- (294,130.2) ;
\draw [color={rgb, 255:red, 74; green, 74; blue, 74 }  ,draw opacity=1 ]   (294,130.2) -- (301.4,122.87) ;
\draw [color={rgb, 255:red, 74; green, 74; blue, 74 }  ,draw opacity=1 ]   (317.4,139.87) -- (329.4,142.87) ;
\draw [color={rgb, 255:red, 74; green, 74; blue, 74 }  ,draw opacity=1 ]   (317.4,139.87) -- (317.4,129.87) ;
\draw [color={rgb, 255:red, 74; green, 74; blue, 74 }  ,draw opacity=1 ]   (274.2,128) -- (285.4,125.87) ;
\draw [color={rgb, 255:red, 74; green, 74; blue, 74 }  ,draw opacity=1 ]   (285.4,125.87) -- (284.2,115) ;
\draw [color={rgb, 255:red, 74; green, 74; blue, 74 }  ,draw opacity=1 ]   (297.2,111) -- (301.4,122.87) ;
\draw [color={rgb, 255:red, 74; green, 74; blue, 74 }  ,draw opacity=1 ]   (301.4,122.87) -- (309.2,115) ;
\draw [color={rgb, 255:red, 74; green, 74; blue, 74 }  ,draw opacity=1 ]   (317.4,129.87) -- (328.2,123) ;
\draw [color={rgb, 255:red, 74; green, 74; blue, 74 }  ,draw opacity=1 ]   (317.4,129.87) -- (314.2,118) ;
\draw [color={rgb, 255:red, 74; green, 74; blue, 74 }  ,draw opacity=1 ]   (329.4,142.87) -- (333.2,134) ;
\draw [color={rgb, 255:red, 74; green, 74; blue, 74 }  ,draw opacity=1 ]   (329.4,142.87) -- (336.2,149) ;
\draw [color={rgb, 255:red, 74; green, 74; blue, 74 }  ,draw opacity=1 ]   (243.97,170.83) -- (226.19,171.02) ;
\draw [color={rgb, 255:red, 74; green, 74; blue, 74 }  ,draw opacity=1 ]   (245.59,154.75) -- (243.97,170.83) ;
\draw [color={rgb, 255:red, 74; green, 74; blue, 74 }  ,draw opacity=1 ]   (218.37,176.64) -- (226.19,171.02) ;
\draw [color={rgb, 255:red, 74; green, 74; blue, 74 }  ,draw opacity=1 ]   (226.19,171.02) -- (223.12,161.07) ;
\draw [color={rgb, 255:red, 74; green, 74; blue, 74 }  ,draw opacity=1 ]   (245.59,154.75) -- (253.81,145.5) ;
\draw [color={rgb, 255:red, 74; green, 74; blue, 74 }  ,draw opacity=1 ]   (245.59,154.75) -- (236.73,150.12) ;
\draw [color={rgb, 255:red, 74; green, 74; blue, 74 }  ,draw opacity=1 ]   (215.08,187.55) -- (218.37,176.64) ;
\draw [color={rgb, 255:red, 74; green, 74; blue, 74 }  ,draw opacity=1 ]   (218.37,176.64) -- (208.18,172.67) ;
\draw [color={rgb, 255:red, 74; green, 74; blue, 74 }  ,draw opacity=1 ]   (210.65,159.3) -- (223.12,161.07) ;
\draw [color={rgb, 255:red, 74; green, 74; blue, 74 }  ,draw opacity=1 ]   (223.12,161.07) -- (219.75,150.51) ;
\draw [color={rgb, 255:red, 74; green, 74; blue, 74 }  ,draw opacity=1 ]   (236.73,150.12) -- (235.64,137.37) ;
\draw [color={rgb, 255:red, 74; green, 74; blue, 74 }  ,draw opacity=1 ]   (236.73,150.12) -- (224.73,147.47) ;
\draw [color={rgb, 255:red, 74; green, 74; blue, 74 }  ,draw opacity=1 ]   (253.81,145.5) -- (247.71,138.03) ;
\draw [color={rgb, 255:red, 74; green, 74; blue, 74 }  ,draw opacity=1 ]   (253.81,145.5) -- (262.39,142.31) ;
\draw [color={rgb, 255:red, 74; green, 74; blue, 74 }  ,draw opacity=1 ]   (243.68,230.7) -- (241.44,248.33) ;
\draw [color={rgb, 255:red, 74; green, 74; blue, 74 }  ,draw opacity=1 ]   (227.98,226.89) -- (243.68,230.7) ;
\draw [color={rgb, 255:red, 74; green, 74; blue, 74 }  ,draw opacity=1 ]   (245.93,256.85) -- (241.44,248.33) ;
\draw [color={rgb, 255:red, 74; green, 74; blue, 74 }  ,draw opacity=1 ]   (241.44,248.33) -- (231.16,250.02) ;
\draw [color={rgb, 255:red, 74; green, 74; blue, 74 }  ,draw opacity=1 ]   (227.98,226.89) -- (219.94,217.48) ;
\draw [color={rgb, 255:red, 74; green, 74; blue, 74 }  ,draw opacity=1 ]   (227.98,226.89) -- (222.18,235.03) ;
\draw [color={rgb, 255:red, 74; green, 74; blue, 74 }  ,draw opacity=1 ]   (256.3,261.6) -- (245.93,256.85) ;
\draw [color={rgb, 255:red, 74; green, 74; blue, 74 }  ,draw opacity=1 ]   (245.93,256.85) -- (240.61,266.4) ;
\draw [color={rgb, 255:red, 74; green, 74; blue, 74 }  ,draw opacity=1 ]   (227.7,262.12) -- (231.16,250.02) ;
\draw [color={rgb, 255:red, 74; green, 74; blue, 74 }  ,draw opacity=1 ]   (231.16,250.02) -- (220.24,251.9) ;
\draw [color={rgb, 255:red, 74; green, 74; blue, 74 }  ,draw opacity=1 ]   (222.18,235.03) -- (209.4,234.37) ;
\draw [color={rgb, 255:red, 74; green, 74; blue, 74 }  ,draw opacity=1 ]   (222.18,235.03) -- (217.91,246.56) ;
\draw [color={rgb, 255:red, 74; green, 74; blue, 74 }  ,draw opacity=1 ]   (219.94,217.48) -- (211.71,222.51) ;
\draw [color={rgb, 255:red, 74; green, 74; blue, 74 }  ,draw opacity=1 ]   (219.94,217.48) -- (217.96,208.54) ;
\draw [color={rgb, 255:red, 74; green, 74; blue, 74 }  ,draw opacity=1 ]   (342.35,185) -- (351.29,169.83) ;
\draw [color={rgb, 255:red, 74; green, 74; blue, 74 }  ,draw opacity=1 ]   (354.92,195.14) -- (342.35,185) ;
\draw [color={rgb, 255:red, 74; green, 74; blue, 74 }  ,draw opacity=1 ]   (350.61,160.14) -- (351.29,169.83) ;
\draw [color={rgb, 255:red, 74; green, 74; blue, 74 }  ,draw opacity=1 ]   (351.29,169.83) -- (361.16,172.66) ;
\draw [color={rgb, 255:red, 74; green, 74; blue, 74 }  ,draw opacity=1 ]   (354.92,195.14) -- (358.43,207.15) ;
\draw [color={rgb, 255:red, 74; green, 74; blue, 74 }  ,draw opacity=1 ]   (354.92,195.14) -- (363.32,190.16) ;
\draw [color={rgb, 255:red, 74; green, 74; blue, 74 }  ,draw opacity=1 ]   (343.2,151.4) -- (350.61,160.14) ;
\draw [color={rgb, 255:red, 74; green, 74; blue, 74 }  ,draw opacity=1 ]   (350.61,160.14) -- (359.13,153.68) ;
\draw [color={rgb, 255:red, 74; green, 74; blue, 74 }  ,draw opacity=1 ]   (369.02,163.07) -- (361.16,172.66) ;
\draw [color={rgb, 255:red, 74; green, 74; blue, 74 }  ,draw opacity=1 ]   (361.16,172.66) -- (371.68,175.57) ;
\draw [color={rgb, 255:red, 74; green, 74; blue, 74 }  ,draw opacity=1 ]   (363.32,190.16) -- (374.5,196.2) ;
\draw [color={rgb, 255:red, 74; green, 74; blue, 74 }  ,draw opacity=1 ]   (363.32,190.16) -- (371.67,181.44) ;
\draw [color={rgb, 255:red, 74; green, 74; blue, 74 }  ,draw opacity=1 ]   (358.43,207.15) -- (367.78,206.06) ;
\draw [color={rgb, 255:red, 74; green, 74; blue, 74 }  ,draw opacity=1 ]   (358.43,207.15) -- (356.69,216.16) ;
\draw [color={rgb, 255:red, 74; green, 74; blue, 74 }  ,draw opacity=1 ]   (305.51,243.43) -- (322.3,249.26) ;
\draw [color={rgb, 255:red, 74; green, 74; blue, 74 }  ,draw opacity=1 ]   (298.54,258.01) -- (305.51,243.43) ;
\draw [color={rgb, 255:red, 74; green, 74; blue, 74 }  ,draw opacity=1 ]   (331.56,246.63) -- (322.3,249.26) ;
\draw [color={rgb, 255:red, 74; green, 74; blue, 74 }  ,draw opacity=1 ]   (322.3,249.26) -- (321.83,259.67) ;
\draw [color={rgb, 255:red, 74; green, 74; blue, 74 }  ,draw opacity=1 ]   (298.54,258.01) -- (287.68,263.93) ;
\draw [color={rgb, 255:red, 74; green, 74; blue, 74 }  ,draw opacity=1 ]   (298.54,258.01) -- (305.32,265.36) ;
\draw [color={rgb, 255:red, 74; green, 74; blue, 74 }  ,draw opacity=1 ]   (338.36,237.47) -- (331.56,246.63) ;
\draw [color={rgb, 255:red, 74; green, 74; blue, 74 }  ,draw opacity=1 ]   (331.56,246.63) -- (339.81,253.8) ;
\draw [color={rgb, 255:red, 74; green, 74; blue, 74 }  ,draw opacity=1 ]   (332.96,265.55) -- (321.83,259.67) ;
\draw [color={rgb, 255:red, 74; green, 74; blue, 74 }  ,draw opacity=1 ]   (321.83,259.67) -- (321.42,270.74) ;
\draw [color={rgb, 255:red, 74; green, 74; blue, 74 }  ,draw opacity=1 ]   (305.32,265.36) -- (302.03,277.73) ;
\draw [color={rgb, 255:red, 74; green, 74; blue, 74 }  ,draw opacity=1 ]   (305.32,265.36) -- (315.71,271.92) ;
\draw [color={rgb, 255:red, 74; green, 74; blue, 74 }  ,draw opacity=1 ]   (287.68,263.93) -- (290.9,273.03) ;
\draw [color={rgb, 255:red, 74; green, 74; blue, 74 }  ,draw opacity=1 ]   (287.68,263.93) -- (278.53,264.03) ;
\draw  [fill={rgb, 255:red, 0; green, 0; blue, 0 }  ,fill opacity=1 ] (300.48,145.87) .. controls (300.4,144.8) and (301.2,143.94) .. (302.26,143.94) .. controls (303.32,143.94) and (304.25,144.8) .. (304.32,145.87) .. controls (304.4,146.93) and (303.6,147.79) .. (302.54,147.79) .. controls (301.48,147.79) and (300.55,146.93) .. (300.48,145.87) -- cycle ;
\draw  [fill={rgb, 255:red, 155; green, 155; blue, 155 }  ,fill opacity=0.08 ] (206,115) .. controls (219.2,97.8) and (227.2,92.8) .. (271.2,89.8) .. controls (282.62,89.02) and (294.06,89.49) .. (304.99,90.7) .. controls (336.16,94.16) and (363.11,103.7) .. (373.2,107.8) .. controls (386.83,113.34) and (428.34,123.48) .. (404.77,163.14) .. controls (381.2,202.8) and (423.2,230.8) .. (405.2,255.8) .. controls (387.2,280.8) and (345.2,298.8) .. (297.34,286.93) .. controls (249.48,275.07) and (184.13,293.1) .. (174.2,274.8) .. controls (164.27,256.5) and (168.69,240.88) .. (173.15,217.17) .. controls (177.6,193.45) and (192.8,132.2) .. (206,115) -- cycle ;
\draw  [fill={rgb, 255:red, 0; green, 0; blue, 0 }  ,fill opacity=1 ] (283.48,125.87) .. controls (283.4,124.8) and (284.2,123.94) .. (285.26,123.94) .. controls (286.32,123.94) and (287.25,124.8) .. (287.32,125.87) .. controls (287.4,126.93) and (286.6,127.79) .. (285.54,127.79) .. controls (284.48,127.79) and (283.55,126.93) .. (283.48,125.87) -- cycle ;
\draw    (267.67,73.27) .. controls (278.17,69.45) and (286.32,101.12) .. (292.26,120.52) ;
\draw [shift={(293.09,123.18)}, rotate = 252.47] [fill={rgb, 255:red, 0; green, 0; blue, 0 }  ][line width=0.08]  [draw opacity=0] (8.93,-4.29) -- (0,0) -- (8.93,4.29) -- cycle    ;

\draw (274,192.4) node [anchor=north west][inner sep=0.75pt]    {$A( t)$};
\draw (293.93,150.44) node [anchor=north west][inner sep=0.75pt]    {$x$};
\draw (338,119.4) node [anchor=north west][inner sep=0.75pt]    {$T_{x}( \varepsilon )$};
\draw (348,232.4) node [anchor=north west][inner sep=0.75pt]    {$A( t+\varepsilon )$};
\draw (250,61.8) node [anchor=north west][inner sep=0.75pt]    {$\gamma '$};

\end{tikzpicture}
\caption{
An illustration of how we couple sites in $A(t+\epsi)-A(t)$ with nodes in $\coprod_{x\in\del A(t)}T_x$. For every $x\in \del A(t)$, a tree $T_x$ rooted at $x$ is mapped into $A(t+\epsi)$. A path $\gamma':[0,1]\to \coprod_{x\in \del A(t)}T_x$ is illustrated in the figure, identified with $f(\gamma'):[0,1]\to G$. The same two sites in $G$ can be connected by multiple paths in different trees, since $f:\coprod_{x\in \del A(t)} T_x\to G$ is not injective.  However, each path from some $x \in \del A(t)$ to another vertex $y$ is associated with a unique path in $T_x$ rooted at $x$.
}
\label{coupling tree}
\end{figure}
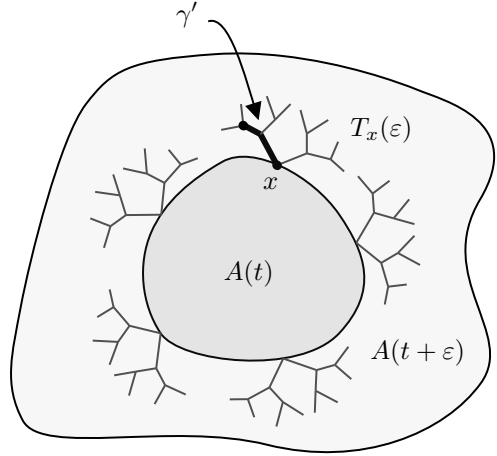

    Now we assign passage times to sites in $G \setminus A(t)$.  To each $x \in \del A(t)$ we assign the passage time of $o_x$.  To define the passage time for other vertices, we first give a couple of definitions:
    \begin{itemize}
    \item The \deffont{passage time along a path} in a graph is the sum of the passage times of all the vertices of the path, including endpoints.
    \item Given a vertex $y \in G \setminus A(t)$, the \deffont{passage time from $A(t)$ to $y$} is the minimal passage time along any path from any vertex of $\del A(t)$ to $y$.
    \item Given a vertex $z \in \coprod_x T_x$, the \deffont{passage time to just before $z$} is the sum of the passage times of the vertices along the unique simple path from $x$ to $z$, \deffont{not} including $z$.
    \end{itemize}
    Then to every remaining $y \in G \setminus A(t)$, we assign the passage time of $\tilde y$, where $\tilde y$ minimizes the passage time to just before $\tilde y$ among all vertices in $f^{-1}(y)$.

    Since the choice of $\tilde y$ does not depend on the passage time of $\tilde y$, the resulting passage time of $y$ is exponentially distributed with mean $1$ and independent of all other passage times.  Therefore this indeed defines a coupling.

    \begin{lem}\label{lem 3.11}
        For every $y \in G \setminus A(t)$, if the passage time from $A(t)$ to $y$ is $\rho$, then there is a $\tilde y \in f^{-1}(y)$ such that the passage time from $\tilde y$ to the root of the tree containing it is at most $\rho$.
    \end{lem}
    \begin{proof}
        We prove this for every $y \in G \setminus A(t)$ by induction on the passage time from $A(t)$ to $y$.  For every site $x \in \partial A(t)$, its passage time is the passage time of $o_x$, so the statement is true.  Now let $y \in G \setminus A(t)$ be another site, let $\tilde y \in T_x$ be the site that it is coupled to and consider the path which realizes the passage time from $A(t)$ to $y$ (the ``shortest path'').  By induction, the passage time from $A(t)$ to the second-to-last vertex of this path is greater than the passage time to just before $\tilde y$.  Therefore, by the definition of the coupling, the passage time from $A(t)$ to $y$ is greater than the passage time from $o_x$ to $\tilde y$.
    \end{proof}
    In particular, if $y \in A(t+\epsi) \setminus A(t)$, then there is some $T_x$ and some $\tilde y \in f^{-1}(y) \cap T_x$ such that the passage time from $o_x$ to $\tilde y$ is at most $\epsi$.  Therefore $\tilde y$ is contained in the subtree $T_x(\epsi)$ of $T_x$, which is again the maximal subtree of $T_x$ containing the root whose nodes all have passage time less than $\epsi$.  It follows that
    \[A(t+\epsi) \setminus A(t) \subseteq f\Bigl(\coprod\nolimits_{x \in \del A(t)} T_x(\epsi)\Bigr).\]
    We determined above that $\mathbb E(|T_x(\epsi)|)=\frac{1}{2}$.  Since the $|T_x(\epsi)|$ are i.i.d.\ random variables, the law of large numbers implies that with high probability as $t \to \infty$,
    \[\Bigl\lvert \coprod\nolimits_{x \in \del A(t)} T_x(\epsi) \Bigr\rvert \leq \lvert\del A(t)\rvert.\]
    This proves \eqref{A(t+epsi)<<A(t)+del A(t)}.
\end{proof}

By Lemma \ref{enough ball lemma}, at time $t-2$, we can find $N \geq C_1F(|A(t-2)|)$ disjoint $R$-balls $B_1,\ldots,B_N$ that are adjacent to but do not intersect $A(t-2)$. By Lemma \ref{generalized lem 9},
\[N \geq C_1F(|A(t-2)|) \geq C_1C_2^{-1}F(|A(t)|).\]

\begin{lem}\label{enough fraction ball lemma}
Let $S$ be as in the statement of Theorem \ref{generalized theorem 8}.  There is a $C_3=C_3(R,d)>0$ such that with high probability as $t \to \infty$, the number of values of $j$, for $1\leq j\leq N$, such that $B_j\cap A(t)$ is isomorphic to $S$ is at least $C_3(R,d)C_1C_2^{-1}F(|A(t)|)$. Here, the $B_j$ are disjoint $R$-balls $B_1,\ldots,B_N$ that are adjacent to but do not intersect $A(t-2)$, whose existence is guaranteed by Lemma \ref{enough ball lemma}.
\end{lem}

\begin{proof}
Let $p$ be a site that is not contained in $A(t-2)$. Denote its passage time by $\rho_p$. By the memorylessness of the exponential distribution, these passage times $\rho_p$ are all i.i.d.  Let $S_j$ be a translate of $S$ inside $B_j$.  Now let $X_j$ be the event that for all $p\in B_j$,
\begin{align*}
    \rho_p&\leq d^{-R-1}\equiv \epsi  &\text{if }p\in S_j\\
    \rho_p&\geq 3 &\text{if }p\in B_j-S.
\end{align*}
Then all the $X_j$ are also i.i.d. and the probability that $X_j$ occurs is positive since
\begin{align*}
    \mathbb{P}(X_j)&=\mathbb{P}(\rho_p\leq \epsi)^{|S|}\cdot \mathbb{P}(\rho_p\geq 3)^{|B_j-S_j|}\\
    &\geq \min\qty{\mathbb{P}(\rho_p\leq \epsi),\mathbb{P}(\rho_p\geq 3)}^{|B_j|}\\
    &\geq \min\qty{\mathbb{P}(\rho_p\leq \epsi),\mathbb{P}(\rho_p\geq 3)}^{d^{R+1}}>0.
\end{align*}
Taking $C_3(R,d)=\frac{1}{2}\min\qty{\mathbb{P}(\rho_p\leq \epsi),\mathbb{P}(\rho_p\geq 3)}^{d^{R+1}}>0$ as our constant, by the law of large numbers, with high probability we know as least $C_3(R,d)C_1C_2^{-1}F(|A(t)|)$ of the $X_j$ will occur.

Suppose $X_j$ occurs. Let $p$ be a vertex of $A(t-2)$ which is adjacent to $B_j$.  Since $S_j$ is connected and (since it contains the whole outer layer of $B_j$) contains a vertex adjacent to $p$, every point of $S_j$ is connected to $p$ by a path $\gamma$ of length at most $d^{R+1}-1$. So the time it takes for all points in $S$ to get infected is at most
\[
\sum_{q\in \gamma}\rho_q\leq 
\ell(\gamma)\epsi\leq (d^{R+1}-1)\epsi=\frac{d^{R+1}-1}{d^{R+1}}<1.\]
Therefore, for $t-1<s<t+1$, we know $A(s)$ will contain all points of $S_j$ and no points in $B_j-S_j$. In particular, when $s=t$, we know $A(t)$ will contain all points of $S_j$ and no points of $B_j-S_j$. Since we have shown that at least $C_3C_1C_2^{-1}F(|A(t)|)$ of such events will occur with high probability, this concludes the proof.
\end{proof}

Taking $C=C_3C_1C_2^{-1}$ completes the proof of Theorem \ref{generalized theorem 8}.

\section{The Eden model on hyperbolic tessellations}\label{The Eden model on hyperbolic tessellations}
We have established our key technical result Theorem \ref{generalized theorem 8}. In the following two sections we demonstrate how Theorem \ref{generalized theorem 8} can be applied to obtain topological information about the Eden model on hyperbolic tessellations, parallel the results of \cite{MRS}.  Later, we will give a further generalization of these results.

\subsection*{General setup}
Denote the $n$-dimensional hyperbolic space by $\bb H^n$.  A \deffont{tessellation} of $\bb H^n$ is a covering by a finite set of isometry classes of \deffont{tiles}: compact sets $K_\alpha$ with nonempty interior such that a point is in the boundary of $K_\alpha$ if and only if it is shared with some other $K_\beta$.

We construct a large family of regular tessellations (ones using one isometry type of cell) using Voronoi cells of orbits of lattices.  Let $\Gamma \subset \operatorname{Isom}(\bb H^n)$ be a cocompact lattice.
\begin{lem}\label{trivialstab}
  There is a point $x \in \bb H^n$ with trivial stabilizer with respect to the action of $\Gamma$.
\end{lem}
\begin{proof}
  For every non-identity element $g \in \Gamma$, let $X_g$ be its fixed set. By Lemma \ref{hyp fixed set}, each $X_g$ is either empty or a hyperbolic subspace of lower dimension.  Since $\Gamma$ is countable, the union of all the $X_g$ is a strict subset of $\bb H^n$.
\end{proof}
Now pick a point $x \in \bb H^n$ with trivial stabilizer, and consider the orbit $\ms O_x$ of $x$ under the action of $\Gamma$. The \deffont{Voronoi cell} of an element $x_\alpha \in \ms{O}_x$ is the set
\[V_\alpha=\bigl\{y\in \bb{H}^n: \forall x_\beta\in \ms{O}_x\;d(y,x_\alpha)\leq d(y,x_\beta)\bigr\}.\]
Clearly, the Voronoi cells associated to $\ms O_x$ are closed and isometric, since they are translates of each other under the action of $\Gamma$.  We will show that they are compact, convex, and the tiles of a tessellation.

\begin{lem}
  The Voronoi cells associated to $\ms O_x$ have bounded diameter, hence they are compact.
\end{lem}
\begin{proof}
  Since $\Gamma$ is a cocompact lattice, by definition there is a compact set $K \subset \bb H^n$ such that $\bigcup_{g \in \Gamma} g \cdot K=\bb H^n$.  In particular, $h \cdot x \in K$ for some $h \in \Gamma$.  Then $K$ has bounded diameter $D$.  Since every point $y \in \bb H^n$ is contained in $g \cdot K$ for some $g \in \Gamma$, $y$ is at distance at most $D$ from $gh \cdot x$.  Therefore every point in the Voronoi cell $V_\alpha$ is at distance at most $D$ from $x_\alpha$.  By the triangle inequality, $\diam(V_\alpha) \leq 2D$.
\end{proof}

\begin{lem} \label{Voronoi-tessellate}
Voronoi cells associated with $\ms O_x$ are the tiles of a tessellation.
\end{lem}

\begin{proof}
Let $\{V_\alpha\}$ be the set of Voronoi cells associated with $\ms O_x$.  It suffices to show two things:
\begin{enumerate}
\item The $V_\alpha$ cover $\bb H^n$.
\item If a point $y \in \bb H^n$ is contained in $V_\alpha \cap V_\beta$ for some $\alpha \neq \beta$, then $y \in \del V_\beta$.
\end{enumerate}

To see (1), notice that $\ms O_x$ is a closed set.  Therefore for any $y \in \td X$ there is a closest point of $\ms O_x$.

To prove (2), consider a point $y \in V_\alpha \cap V_\beta$ and consider a minimal geodesic $\gamma:[0,1] \to \td X$ from $y$ to $x_\alpha$.  We claim that $\gamma((0,1])$ does not intersect $V_\beta$.  Indeed, suppose that $z=\gamma(t)$ were contained in $V_\beta$.  Then $d(z,x_\beta) \leq d(z,x_\alpha)$.

Suppose first that $d(z,x_\beta)<d(z,x_\alpha)$.  But this would then mean that
\[d(y,x_\beta) \leq d(y,z)+d(z,x_\beta)<d(y,z)+d(z,x_\alpha)=d(y,x_\alpha),\]
which contradicts the fact that $y \in V_\alpha$.

Now suppose that $d(z,x_\beta)=d(z,x_\alpha)$.  But then concatenating $\gamma|_{[0,t]}$ with a minimal geodesic from $z$ to $x_\beta$ would give a shortest path from $y$ to $x_\beta$.  Since $\bb H^n$ is a complete Riemannian manifold, this shortest path is a geodesic.  But since $\gamma|_{[0,t]}$ is an initial segment of this geodesic, the geodesic is in fact $\gamma$ and $x_\alpha=x_\beta$, again a contradiction.

Since the sequence $\{\gamma(1/n)\}_{n \in \bb N}$ converges to $y$ and consists of points outside $V_\beta$, $y$ is in the boundary of $V_\beta$.
%
%
\end{proof}

\begin{defn}
A set $\Omega \subseteq \bb{H}^n$ is said to be \deffont{convex} if for any pair $x,y\in\Omega$, the unique geodesic connecting $x$ to $y$ is contained in $\Omega$.
\end{defn}

\begin{defn}
Let $p,q\in\bb{H}^n$. Then their \deffont{bisector} $B(p,q)$ is the set
\[B(p,q)=\bigl\{x\in\bb{H}^n:d(p,x)=d(q,x)\bigr\}.\]
\end{defn}

\begin{lem}\label{H^n, bisec=geo}
Let $p,q\in\bb{H}^n$. Then $B(p,q)$ is a totally geodesic subspace.
\end{lem}

\begin{proof}
We use the upper half-plane model of $\bb{H}^n$. We can always apply an elliptic isometry to rotate $p$ around $q$ so that they differ only in the $x_1$-coordinate. Then $B(p,q)$ is an $n-1$-dimensional hyperplane with constant $x_1$-coordinate, which is a totally geodesic subspace in $\bb{H}^n$. Since isometries preserve geodesics, this concludes the proof.
\end{proof}

In \cite{Beem}, it is shown that the converse of this statement is also essentially true: if the bisector of any pair of points in a Riemannian (or pseudo-Riemannian) manifold $M$ is totally geodesic, then $M$ has constant curvature.

The next result tells us that the Voronoi cells constructed above are convex when $\td{X}=\bb{H}^n$:

\begin{lem}
Suppose that $\{V_\alpha\}$ is the set of Voronoi cells generated by a discrete set of points $\{x_\alpha\}$ in $\bb{H}^n$.  Then any nonempty intersection of the $V_\alpha$ is a convex set, and therefore contractible.
\end{lem}

\begin{proof}
    Since the intersection of convex sets is always convex, it suffices to prove that a single $V_\alpha$ is convex.  Now notice that
    \[V_\alpha=\bigcap_{\beta \neq \alpha} \{y \in \bb H^n: d(y,x_\alpha) \leq d(y,x_\beta)\}.\]
    Write $H_{\alpha,\beta}=\{y \in \bb H^n: d(y,x_\alpha) \leq d(y,x_\beta)\}$.  It suffices to prove that each $H_{\alpha,\beta}$ is convex.

    To do this, we use the convexity of $B(x_\alpha,x_\beta)=\partial H_{\alpha,\beta}$.  Suppose that the shortest path $\gamma:[0,1] \to \bb H^n$ between $y,z \in H_{\alpha,\beta}$ has a point outside of $H_{\alpha,\beta}$, say $\gamma(t)$.  Then there must be two points $s<t<s'$ such that $\gamma(s),\gamma(s') \in B(x_\alpha,x_\beta)$.  But then the shortest path between $\gamma(s)$ and $\gamma(s')$ lies in $B(x_\alpha,x_\beta)$.  Substituting this for the portion of $\gamma$ between $s$ and $s'$ creates a path between $y$ and $z$ which is shorter than $\gamma$, a contradiction.
\end{proof}

\subsection*{Hyperbolic Eden models}
We have shown that Voronoi cells of an orbit $\ms{O}_x$ of an action of $\Gamma$ on $\bb{H}^n$ are convex fundamental domains. The generalization of the Eden model on a tessellation of $\bb{H}^n$ is then clear: pick a tile to start with; then at every time step, randomly choose a tile adjacent to the infected region to infect next.  Here we say two tiles are \deffont{adjacent} if they share an $(n-1)$-dimensional face. We will use $A(t)$ to denote the Eden model at time $t$.

We can view each tile $V_\alpha$ as the vertex $\alpha$ of a graph $T$, and we stipulate that there is an edge from $\alpha$ to $\beta$ in $T$ if and only if 
\[\dim(V_\alpha\cap V_\beta)=n-1\]
in $\bb{H}^n$.
\begin{figure}
    \centering
    \includegraphics[scale=0.2]{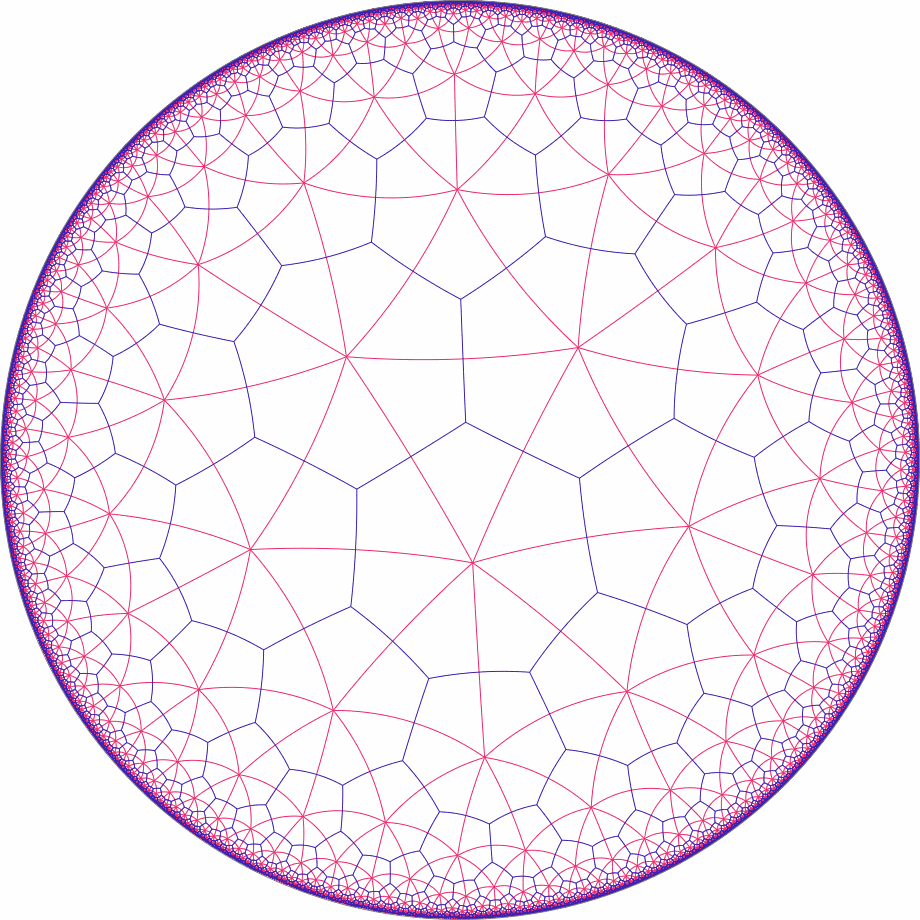}
    \caption{
    An example of a tessellation of $\bb H^n$ by Voronoi cells is the regular hyperbolic tiling of type $(n,m)$, which is a tiling where $m$ regular $n$-gons share a vertex. This image by \href{https://dmitrybrant.com/2007/01/24/hyperbolic-tessellations}{Dmitry Brant} shows a $(7,3)$ regular tiling in blue and its dual graph $T$ in red.
    }
    \label{hyptil}
\end{figure}

\begin{lem}\label{ball H^n intersect finite cells}
A closed ball in $\bb{H}^n$ intersects finitely many compact tiles of any tessellation.
\end{lem}

\begin{proof}
Let $d$ denote the maximum diameter of all the tiles, since there are finitely many isometry classes.  Similarly, let $v$ denote the minimal volume of a tile.  If a closed ball $B(y,r)$ intersects a tile $V$, then the ball $B(y+d,r)$ contains $V$.  Therefore the ball $B(y,r)$ contains at most $v^{-1}\vol B(y+d,r)$ tiles.
\end{proof}

We will also use the following well-known result to prove that $T$ is a Cayley graph for $\Gamma$:

\begin{lem}[Thom's transversality theorem] \label{thom transv}
Let $P$ be a smooth manifold and $M$ be a smooth submanifold of manifold $N$. Let $C^\infty_{\pitchfork_M}(P,N)$ denote the set of all smooth maps $P\to N$ that are transversal to $M$. Then $C^\infty_{\pitchfork_M}(P,N)$ is a dense subspace of $C(P,N)$ with the compact-open topology.
\end{lem}

This result was originally obtained in \cite{Thom1954}. The proof of a slightly weaker version can be found in \cite[chapter IV.2]{kosinski1992}.

\begin{lem} \label{lem:Cayley}
  Suppose that $\Gamma$ acts transitively on the tiles of a tessellation by convex, compact tiles.  Then the graph $T$ is a Cayley graph for $\Gamma$.
\end{lem}
In particular, this means that $T$ is infinite, locally finite, connected, and vertex-transitive.
\begin{proof}
  By Lemma \ref{recognizeCayley}, it suffices to prove that $T$ is locally finite and connected and $\Gamma$ acts freely and transitively on $T$ by graph isomorphisms.

  The action of $\Gamma$ on $T$ is induced by an action by isometries on $\bb H^n$, which takes $(n-1)$-dimensional intersections to $(n-1)$-dimensional intersections.  Therefore $\Gamma$ acts by graph homomorphisms.  Moreover, the action is free since we chose $x$ to have trivial stabilizer and transitive since the centers of cells are exactly the orbit of $x$ under $\Gamma$.

  By Lemma \ref{ball H^n intersect finite cells}, there are finitely many tiles that intersect any closed ball, hence $T$ is locally finite.
  
  To show that $T$ is connected, let us pick $\alpha,\beta\in T$. We want to show that there is a path in $T$ that starts at $\alpha$ and ends at $\beta$. This amounts to showing that in $\mathbb{H}^n$, we can find a path starting in $V_\alpha$ and ending in $V_\beta$ that only travels between adjacent tiles, i.e., it avoids any intersections of dimension $\leq n-2$. Since $\mathbb{H}^n$ is path-connected, we can certainly find a smooth path $\gamma:[0,1] \to \bb H^n$ starting in $V_\alpha$ and ending in $V_\beta$. Take an open ball $B$ that contains $V_\alpha$, $V_\beta$, and the image of $\gamma$. By Lemma \ref{ball H^n intersect finite cells}, there are finitely many intersections of dimension $\leq n-2$ inside $B$, each of which is contained in a hyperplane $\mc{H}^{n-2}$ of dimension $n-2$.
  
  Now we claim that we can perturb the path $\gamma$ to a new path $\gamma'$ in $B$ which is transversal to the finitely many hyperplanes of dimension $n-2$ mentioned above, and whose endpoints still lie in $V_\alpha$ and $V_\beta$. Since for every point $u\in \gamma'\cap\mc{H}^{n-2}$, we have
  \[\dim([0,1])+\dim \mc{H}^{n-2}=n-1<n=\dim \bb{H}^n,\]
  transversality implies that $\gamma'\cap\mc{H}^{n-2}=\emptyset$, so $\gamma'$ is the path that we want. Let
  \[\epsi=\min\{d(\gamma(0),\partial V_\alpha),d(\gamma(1),\partial V_\beta),d(\gamma([0,1]),\partial B)\}.\]
  By Lemma \ref{thom transv}, we may perturb $\gamma$ by less than $\epsi/2$ to a new path $\gamma^{(1)}$ such that it is transversal to the first hyperplane, i.e.\ $\gamma^{(1)}\pitchfork \mc{H}^{n-2}_{1}$. Then, inductively, at $j$-th step, let $d_j=\min\{d(\gamma^{(j-1)},\mc{H}^{n-2}_1),...,d(\gamma^{(j-1)},\mc{H}^{n-2}_{j-1})\}$. We perturb $\gamma^{(j-1)}$ by less than $\min(d_j,\epsi/2^{j})$ such that $\gamma^{(j)}\pitchfork\mc{H}^{n-2}_j$. Note that we can perturb $\gamma^{(j-1)}$ in such that way that it stays transversal to all the previous hyperplanes, i.e., $\gamma^{(j)}\pitchfork\mc{H}^{n-2}_k$ for all $k\leq j$. Since there are only finitely many such hyperplanes, say $N$ of them, this process must terminate. In the end, we will obtain a path $\gamma'$ such that $\gamma'\pitchfork \mc{H}^{n-2}_k$ for all $1\leq k\leq N$, and we can control the total perturbation by
  \[d(\gamma,\gamma')<\epsi\qty(\frac{1}{2}+\frac{1}{4}+\frac{1}{8}+\cdots+\frac{1}{2^N})<\epsi.\]
  Hence the starting point of $\gamma'$ is in $V_\alpha$ and the ending point of $\gamma'$ is in $V_\beta$, as desired.  Moreover, since the image of $\gamma'$ still lies in $B$, it does not hit any new $(n-2)$-dimensional intersections.
\end{proof}
By the Milnor--Schwarz lemma, it follows that the map $T \to \bb H^n$ which sends each $\alpha \mapsto x_\alpha$ is a quasi-isometry.  Moreover, since $\Gamma$ acts on hyperbolic space, it is non-amenable by Example \ref{IP examples}(c), that is, its isoperimetric profile function is bounded below by a linear function.

The main result of this section is:

\begin{thm}\label{H^n betti bound}
Let $A(t)$ be the Eden model at time $t$ on a Voronoi tessellation of $\bb{H}^n$ whose centers are an orbit of the action of a lattice $\Gamma$. Let $1\leq k\leq n-1$ and let $\beta_k(t)$ denote the $k$th Betti number of $A(t)$, i.e. $\beta_k(t)=\rank H_k(A(t))$. Then there are constants $C>c>0$ depending on $k$ and the tessellation such that
\[c|A(t)| \leq \beta_k(t) \leq C|A(t)|\]
with high probability as $t\to\infty$.
\end{thm}
\begin{rmk}
  In the proof we use only that we have a regular tessellation of $\bb H^n$ by convex compact tiles.  Thus we prove the slightly more general statement in Theorem \ref{hyperbolic intro}.
\end{rmk}

Here is a rough sketch of the proof.  By Theorem \ref{generalized theorem 8}, there is a constant $C(R,d)>0$ (where $d$ is the degree of the graph $T$ constructed above) so that for any connected subgraph $\mc{S}$ that satisfies certain hypotheses, the intersection of $A(t)$ with at least $C|A(t)|$ disjoint $R$-balls is isomorphic to $\mc{S}$, with high probability as $t\to\infty$.  Each such subgraph $\mc{S}$ will correspond to a union of Voronoi cells in $\bb{H}^n$, which we denote by $S$.  This means that we can find at least $C|A(t)|$ copies of $S$ along the periphery of the Eden model $A(t)$ at time $t$. Then to prove Theorem \ref{H^n betti bound}, it suffices to construct, for each $1 \leq k \leq n-1$, an $S$ such that each copy of $S$ contributes at least $1$ to the $k$th Betti number $\beta_k$.  The next section will be devoted to the construction of such an $S$. 

There is one subtlety.  Theorem \ref{generalized theorem 8} lets us find many copies of the subgraph $\mc S$ in $T$, and moreover (as shown in the proof) there is a graph automorphism of $T$ sending any of these subgraphs to any other.  But in $\bb{H}^n$, we will need copies of $S$ (as geometric realization in $\bb{H}^n$) to be homeomorphic, which cannot be guaranteed by the existence of such a graph automorphism unless the automorphism is induced by the group action of $G$.  Therefore we must use the modification of Theorem \ref{generalized theorem 8} discussed in Remark \ref{rmk:strengthening}.

\subsection*{Upper bound}
The upper bound of Theorem \ref{H^n betti bound} is essentially trivial, in the sense that it holds for every union of tiles, not just those generated by the Eden model.  Consider any union $U$ of $m$ tiles of our tessellation.  The set of tiles forms what is known as a \deffont{good cover} of $U$: each tile as well as any nonempty intersection of tiles is contractible, because our tiles are convex.  This means that the homology of $U$ is the same as the homology of the nerve $N$ of the good cover, since by Leray's nerve theorem the nerve of the good cover $N$ is homotopy equivalent to $U$.  Every tile intersects the same finite number of other tiles, say $r$.  Therefore, for every $k$, the number of $k$-simplices of $N$ is at most ${r \choose k+1}m$.  This is an upper bound for the rank of $H_k(U)$.

\section{Constructing hyperbolic handles}
We now construct $S$ with the properties described above for $1 \leq k \leq n-2$. The case $k = n-1$ involves a simpler $S$, so we address it separately in the proof below. 

We use the Poincar\'e disc model of $\bb{H}^n$. In this model, $\bb{H}^n$ is identified with the open unit disc in $\bb{R}^n$, with metric
\[\dif s^2=\frac{4(\dif x_1^2+\cdots+\dif x_n^2)}{(1-(x_1^2+\cdots+x_n^2))^2}.\]
We first construct a copy of $S$ at the origin. Then we can use a deck transformation to shift $S$.
\subsection*{Notation}
We consider two different metric spaces with distinct, if related, notions of distance.  The first is $\bb H^n$ with the hyperbolic metric $d$ induced by the metric $\dif s^2$.  The second is the graph $T$ induced by the tessellation, equipped with the graph metric $d_T$.  For the rest this section, we establish the following notational conventions:
\begin{itemize}
    \item Any subgraph of the graph $T$ will be denoted by calligraphic letters.  For example, the subgraph of $T$ associated with the hyperbolic handle will be denoted by $\mc{S}$. A closed ball of radius $R$ in the graph metric $d_T$ centered at $x$ will be denoted by $\mc{B}_R (x)$.
    \item For any subgraph $\mc A$ of the graph $T$, we denote the union of the corresponding tiles in $\bb H^n$ by $\langle \mc A \rangle$.
    \item A subset of the hyperbolic space $\bb{H}^n$ will be denoted by ordinary italic letters.  The hyperbolic handle viewed as a subset of $\bb{H}^n$ (a union of tiles) will be denoted by $S$. A closed ball of radius $R$ in the hyperbolic metric (which we denote by $d$ or $d_{\bb{H}^n}$) centered at $x$ will be denoted by $B_R(x)$, and its boundary will be denoted by $\del B_R(x)=S_R(x)$.
\end{itemize}

\subsection*{Construction at the origin}
We assume without loss of generality that the origin is the center of one of our tiles.

Fix $1\leq k\leq n-2$. Choose a $k$-dimensional plane $P^k$ through the origin in $\bb R^n$ and let $P^{n-k}$ be the plane through the origin perpendicular to $P^k$.  These two planes intersect the Poincar\'e disc $\bb{H}^n$. Denote the union of all tiles in $\bb{H}^n$ that intersect $P^k$ by $\cl{P^k}$.  Denote the union of all tiles in $\bb{H}^n$ that intersect $P^{n-k}$ but do not intersect $\cl{P^k}$ by $\cl{P^{n-k}}$. We claim:

\begin{lem} \label{hyperbolic divergence}
There exists $R_1>0$ such that for all vertices $\alpha$ of $T$ satisfying $d_T(\alpha,0)\geq R_1$, if the tile $V_\alpha$ intersects $P^{n-k}$, then $V_\alpha \in \cl{P^{n-k}}$.
\end{lem}

\begin{proof}
Let $u,v\in\bb{R}^n$ be two vectors in $P^k$ and $P^{n-k}$ respectively with Euclidean norm less than 1, so that they live in the Poincar\'e disc $\bb{H}^n$. Their distance in $\bb{H}^n$ is given by
\[d_{\bb{H}^n}(u,v)=2\ln\qty(\frac{\norm{u-v}+\sqrt{\norm{u}^2\norm{v}^2-2\expval{u,v}+1}}{\sqrt{(1-\norm{u}^2)(1-\norm{v}^2)}}).\]
Since $\langle u,v \rangle=0$, as $\norm{u},\norm{v}\to 1$, we have $d(u,v)\to\infty$.  Therefore as $R \to \infty$, the hyperbolic distance between $P^k \setminus B_R(0)$ and $P^{n-k} \setminus B_R(0)$ also goes to infinity.  Since every Voronoi cell has bounded diameter, there is an $R_0>0$ such that whenever $V_{\alpha}$ contains a point with hyperbolic distance at least $R_{0}$ from $0$ and $V_\alpha$ intersects $P^{n-k}$, $V_\alpha$ does not intersect $P^k$.  The Milnor--Schwarz lemma gives us an quasi-isometry between $T$ and $\bb{H}^{n}$, so taking the constants $C > 1, K > 0$ from the quasi-isometry we can set $R_1=CR_0+K$.
\end{proof}

Consider the ball $\mc B_{R_1}(0)$ in the graph metric.  By the proof of the lemma, the union of tiles $\langle \mc{B}_{R_1}(0) \rangle \cap (\cl{P^k})^c$ contains the $(n-k-1)$-sphere $\Omega$ of radius $R_0$ (in the hyperbolic metric) inside $P^{n-k}$.

\begin{lem}\label{layers of balls}
Let $\epsi>0$ be given. There are constants $R_2,R_3>0$ such that
\[N_\epsi(S_{R_2}(0))\subset \langle \mc{B}_{R_3}(0) \setminus \mc{B}_{R_1}(0) \rangle\]
where $N_\epsi(K)$ denotes the closed $\epsi$-neighborhood of $K$.
\end{lem}

\begin{proof}
The lemma says, more or less, that one can nest the $R_1$-ball in the graph metric inside an $R_2$-ball in the hyperbolic metric, which in turn is nested inside an $R_3$-ball in the graph metric, with some wiggle room to spare (see Figure \ref{hyperbolic handle figure}).  This follows easily from the fact that the inclusion $T \hookrightarrow \bb H^n$ is a quasi-isometry.

Denote $\diam(V_\alpha)=\lambda<\infty$.  There are $C$ and $K$ such that for every point $x \in \langle \mc{B}_{R_1}(0) \rangle$,
\[d(x,0) \leq CR_1+K+\lambda.\]
So we can choose $R_2=CR_1+K+\lambda+\epsi + 1$.  Similarly, for every $\alpha$ such that $V_\alpha$ intersects $N_\epsi(S_{R_2}(0))$,
\[d_T(\alpha,0)<R_3=C(R_2+\lambda + \epsi + 1)+K. \qedhere\]
%
%
%
%
\end{proof}

Now we are ready to finish our construction of $S$. Let $J$ denote the union of all Voronoi cells contained inside the open $R_{2}$-ball around $0$, excluding those Voronoi cells which also intersect $P^{k}$. 
Then, we define 
\[S=\operatorname{closure}(\langle\mc{B}_{R_3}(0)\rangle \setminus J).\]
This set splits into a ``spherical shell'' outside the hyperbolic $R_2$-ball and a ``handle'' inside. Figure \ref{hyperbolic handle figure} gives an illustration of this construction. The reason we take the closure at the end is because when we remove $J$, we remove some parts of boundaries of Voronoi cells in $\langle\mc{B}_{R_3}(0)\rangle$ as well. This way, $S$ corresponds precisely to $\langle \mc S \rangle$ for some subgraph $\mc S$ of $T$. This subgraph $\mc S$ satisfies the conditions in Theorem \ref{generalized theorem 8}:

\begin{figure}
\centering

\tikzset{every picture/.style={line width=0.75pt}} 

\tikzset{every picture/.style={line width=0.75pt}} 

\begin{tikzpicture}[x=0.75pt,y=0.75pt,yscale=-1,xscale=1]
\clip (165,333) rectangle (461,48);
\draw  [fill={rgb, 255:red, 132; green, 132; blue, 132 }  ,fill opacity=0.21, line width=1.5] (297,45) rectangle (341,335);
\draw  [fill={rgb, 255:red, 200; green, 200; blue, 200 }  ,fill opacity=0.66, line width=1.5, even odd rule]  (221.29,93.9) .. controls (223.31,53.21) and (260.68,75.27) .. (296.74,72.68) .. controls (343.97,69.29) and (371.36,65.32) .. (404.58,82.19) .. controls (437.8,95.22) and (428.42,104.42) .. (446.38,164.07) .. controls (460.92,203.33) and (453.9,238.29) .. (450.43,254.24) .. controls (444.65,268.34) and (446.89,300.12) .. (392.07,311.28) .. controls (351.86,319.47) and (269.01,313.73) .. (258.03,308.83) .. controls (226.83,299.01) and (177.09,307.6) .. (170.21,242.58) .. controls (167.7,188) and (189.43,190.88) .. (201.64,157.48) .. controls (214.03,123.59) and (224.03,132.18) .. (221.29,93.9) -- cycle (275.31,132.94) .. controls (289.03,122.98) and (365.58,129.88) .. (371.36,149.81) .. controls (377.14,169.75) and (349.7,201.18) .. (364.14,224.18) .. controls (378.58,247.18) and (283.25,261.75) .. (268.81,238.75) .. controls (254.36,215.75) and (261.59,142.91) .. (275.31,132.94) -- cycle ;
\draw  [dash pattern={on 0.84pt off 2.51pt}] (216.81,191.83) .. controls (216.81,132.97) and (261.75,85.25) .. (317.2,85.25) .. controls (372.64,85.25) and (417.58,132.97) .. (417.58,191.83) .. controls (417.58,250.69) and (372.64,298.4) .. (317.2,298.4) .. controls (261.75,298.4) and (216.81,250.69) .. (216.81,191.83) -- cycle ;
\draw  [line width=1.5]  (226.42,190.6) .. controls (226.42,138.05) and (266.54,95.45) .. (316.04,95.45) .. controls (365.54,95.45) and (405.67,138.05) .. (405.67,190.6) .. controls (405.67,243.15) and (365.54,285.75) .. (316.04,285.75) .. controls (266.54,285.75) and (226.42,243.15) .. (226.42,190.6) -- cycle ;
\draw  [dash pattern={on 0.84pt off 2.51pt}] (236.24,190.6) .. controls (236.24,143.81) and (271.97,105.88) .. (316.04,105.88) .. controls (360.12,105.88) and (395.84,143.81) .. (395.84,190.6) .. controls (395.84,237.39) and (360.12,275.32) .. (316.04,275.32) .. controls (271.97,275.32) and (236.24,237.39) .. (236.24,190.6) -- cycle ;
\draw [line width=2.25]    (342.04,195.51) .. controls (361.54,218.51) and (277.19,221.88) .. (295.82,196.28) ;
\draw   (288.55,212.48) .. controls (281.33,200.98) and (295.39,182.22) .. (296.83,192.18) .. controls (298.27,202.15) and (338.67,203.69) .. (340.55,190.75) .. controls (342.43,177.82) and (362.95,202.48) .. (346.33,214.78) .. controls (329.71,227.08) and (295.77,223.98) .. (288.55,212.48) -- cycle ;

\draw (163.06,118.55) node [anchor=north west][inner sep=0.75pt]    {$\mathcal{B}_{R_{2}}( 0)$};
\draw (362.7,196.97) node [anchor=north west][inner sep=0.75pt]  [font=\small]  {$\mathcal{B}_{R_{1}}( 0)$};
\draw (178.05,251.49) node [anchor=north west][inner sep=0.75pt]    {$N_{\varepsilon }\left( S^{n-1}\right)$};
\draw (315.11,200.94) node [anchor=north west][inner sep=0.75pt]    {$\omega $};
\draw (270.91,188.67) node [anchor=north west][inner sep=0.75pt]    {$\Omega $};
\draw (237,82.4) node [anchor=north west][inner sep=0.75pt]    {$N_{\varepsilon }^{+}$};
\draw (240,160) node [anchor=north west][inner sep=0.75pt]    {$N_{\varepsilon }^{-}$};
\draw (306,150) node [anchor=north west][inner sep=0.75pt]    {$\overline{P^{k}}$};

\end{tikzpicture}

\caption{The hyperbolic handle $S$.}
\label{hyperbolic handle figure}
\end{figure}
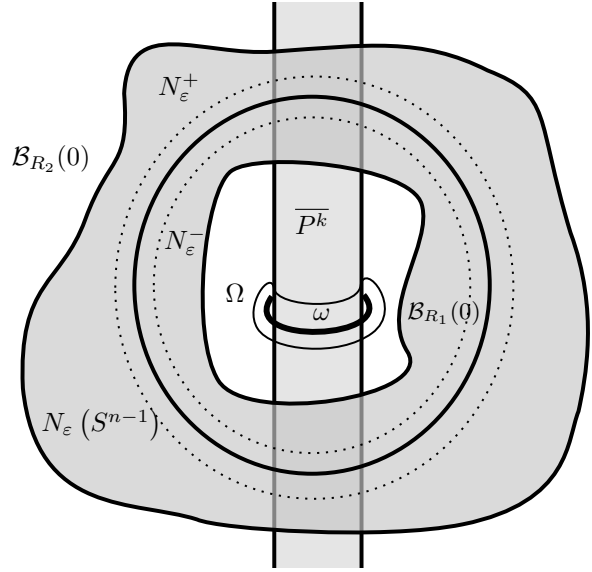
\begin{lem}\label{subgraph satisfies conditions}
The subgraph $\mc S$ is connected, is contained in a ball $\mc B_{R}$ centered at a vertex of $T$, and contains the sphere of the same radius and center. 
\end{lem}
\begin{proof}
    For the latter two conditions it suffices to show that $J \cap \langle \partial \mc B_{R_{3}}(0) \rangle$ is empty, but this is guaranteed by the construction in Lemma \ref{layers of balls}.

    For connectedness, we use the same transversality approach from \ref{lem:Cayley}. First, note that $\langle\mc{B}_{R_3}(0)\rangle \setminus J$ is path-connected, because every point $x$ in this set falls into two categories: either it is a point contained in a Voronoi cell intersecting $S_{R_{2}} \cup P^{k},$ or it is a point inside $\langle\mc{B}_{R_3}(0)\rangle $ outside of $B_{R_{2}}$. In the latter case, taking a path in the graph connecting the vertex corresponding to the Voronoi cell containing $x$ and the origin, we see this path must contain some vertex corresponding to a Voronoi cell intersecting $S_{R_{2}}$ (because this path can be translated to a piecewise geodesic path in hyperbolic space from outside $B_{R_{2}}$ to inside, which must intersect $S_{R_{2}}$).

    Therefore, any point in the second category is connected to a point in the first category by a path, hence it suffices to show points in the first category are connected to each other by paths. For two such points $x$ and $y$, we first connect them to points on $S_{R_{2}} \cup P^{k}$ by paths in the respective Voronoi cells, and then observe that $S_{R_{2}} \cup P^{k}$ is path-connected.

    Now, given two centers of Voronoi cells in $S$ (which are necessarily also contained in $\langle\mc B_{R_3}(0)\rangle \setminus J$), we can take a path between them in $\langle\mc B_{R_3}(0)\rangle \setminus J$. Since $J$ is closed this path has a minimum distance from $J$. Using the argument in $\ref{lem:Cayley}$, we can perturb this path so that it remains in $\langle\mc{B}_{R_3}(0)\rangle \setminus J$, starts and ends in the same Voronoi cells, and avoids all intersections of dimension $\leq n-2$. Translating this to a path in the graph we conclude that $\mc S$ is connected. 
\end{proof}

Now we note that the action of an element of $\Gamma$ induces both an automorphism of $T$ and a homeomorphism of $\bb H^n$ which takes tiles to tiles.  In particular, for every point $x_\alpha \in \ms O_x$, there is an element $g \in \Gamma$ such that $g(0)=x_\alpha$.  Then the action of $g$ takes $S$ to a copy $S_\alpha$ centered at $x_\alpha$, and it takes the subgraph $\mc S$ of $T$ to an isomorphic subgraph centered at $\alpha$.  Applying Theorem \ref{generalized theorem 8}, or rather the variant in Remark \ref{rmk:strengthening}, we get:



\begin{cor}
There is a constant $C=C(n,\Gamma)>0$ such that with high probability, there are at least $C|A(t)|$ disjoint $R$-balls $\mc B_j$ in the graph metric such that $\langle \mc B_j \rangle \cap A(t)$ is homeomorphic to $S$.
\end{cor}

\begin{proof}
Let $\mc A(t)$ be the subgraph of $T$ corresponding to $A(t)$.  By Remark \ref{rmk:strengthening}, there are $C_1(R_3,n)F(|A(t)|)$ disjoint $R$-balls $\mc B_j$ such that $\mc A(t) \cap \mc B_j=g \cdot \mc S$, where $F$ is the isoperimetric profile of $\Gamma$.  The statement of the corollary follows from the fact that $F$ is linear.
%
%
\end{proof}

\subsection*{Homology of the hyperbolic Eden model}
To finish the proof of Theorem \ref{H^n betti bound}, it remains to show that each $S$ adds at least $1$ to $\beta_k(t)$.

We have shown that at time $t$, we can find $C|A(t)|$ homeomorphic copies $S_\alpha$ of the hyperbolic handle, centered at points $\{x_\alpha\}_{\alpha \in J}$. For each such handle $S_\alpha$, denote $N^+_\alpha=S_\alpha \setminus B_{R_2-\epsi}(z_\alpha)$, where $R_2$ is the radius of the hyperbolic sphere $S_{R_2}(0)$ chosen in the course of constructing $S$.  Similarly, denote $N^-_\alpha=S_\alpha \cap B_{R_2+\epsi}(z_\alpha)$.  The overlapping subsets $N^+$ and $N^-$ of $S$, whose shifted copies are the sets defined here, are marked in Figure \ref{hyperbolic handle figure}.  Now define
\[M=A(t) \setminus \coprod_{\alpha \in J} B_{R_2-\epsi}(z_\alpha)
\qquad
N=\coprod_{\alpha \in J} N_\alpha^-.\]
Intuitively, $M$ is the part of $A(t)$ that does not include the handles, plus all disjoint copies of $N^+$, and $N$ is just disjoint copies of $N^-$. Then
\[A(t)=\operatorname{int} M \cup \operatorname{int} N \quad\text{and}\quad M\cap N=\coprod_\alpha N_\epsi(S^{n-1})\simeq \coprod_\alpha S^{n-1}.\]
Therefore, since $1 \leq k \leq n-2$, we have that $H_k(M\cap N) \cong 0$.

Write $A$ for $A(t)$ to simplify notation.  The Mayer--Vietoris theorem gives an exact sequence
\[\cdots \to H_{k}(M\cap N) \to H_k(M)\oplus H_k(N) \xrightarrow{\psi} H_k(A) \to H_{k-1}(M\cap N) \to \cdots,\]
and since $H_k(M \cap N) \cong 0$, the map $\psi$ is injective.  Since $N=\coprod_\alpha(N^-_\alpha)$, we have that $H_k(N) \cong \bigoplus_\alpha H_k(N^-_\alpha)$. Therefore, if we can show that for each $\alpha$, the homology group $H_k(N^-_\alpha)$ has rank at least $1$, this completes the proof of Theorem \ref{H^n betti bound}.  Note that all the $N^-_\alpha$ are isometric, so we just need to understand the homology of the set $N^- \subseteq S$.

Note that $N^{-}$ is contained in $B_{R_{2}+\epsi}(0) \setminus \Omega$ because $\overline{P^k}$ is disjoint from $\Omega$ and $\langle \mathcal B_{R_1}(0) \rangle$ contains $\Omega$. On the other hand, by definition, $N^{-}$ contains $S_{R_{2} + \epsi} \cup \left(B_{R_{2}+\epsi}(0) \cap P^{k} \right)$.

Now, note that $B_{R_{2}+\epsi}(0) \setminus \Omega$ deformation retracts to $S_{R_{2} + \epsi} \cup \left(B_{R_{2}+\epsi}(0) \cap P^{k} \right)$.  (A low-dimensional case is the deformation retraction of a solid ball minus a circle lying on a plane deformation retracting to a sphere union a diameter perpendicular to that plane.  In general as in that case, each point moves out along the line linking it to the closest point on $\Omega$.) This deformation retraction restricts to a retraction from $N^{-}$ to $S_{R_{2} + \epsi} \cup \left(B_{R_{2}+\epsi}(0) \cap P^{k} \right)$, and so the inclusion map from $S_{R_{2} + \epsi} \cup \left(B_{R_{2}+\epsi}(0) \cap P^{k} \right)$ to $N^{-}$ induces an injection on homology. Since our aim is to show that $H_{k}(N_{\alpha}^{-})$ has rank at least $1$, it suffices to do so for $H_{k}(S_{R_{2} + \epsi} \cup \left(B_{R_{2}+\epsi}(0) \cap P^{k} \right))$.

We apply Mayer--Vietoris once more. For some sufficiently small $\epsi' > 0$, let $V$ be an open $\epsi'$-neighborhood of $S_{R_{2} + \epsi}$ in $S_{R_{2} + \epsi} \cup \left(B_{R_{2}+\epsi}(0) \cap P^{k} \right)$, and let $W$ be an open $\epsi'$-neighborhood of $B_{R_{2}+\epsi}(0) \cap P^{k}$ in $S_{R_{2} + \epsi} \cup \left(B_{R_{2}+\epsi}(0) \cap P^{k} \right)$. Then $V, W$ are open in $S_{R_{2} + \epsi} \cup \left(B_{R_{2}+\epsi}(0) \cap P^{k} \right) = V \cup W,$ so we may apply Mayer--Vietoris.

We see that $V$ deformation retracts to $S_{R_{2}+ \epsi} \cong S^{n-1}$, $W$ deformation retracts to a point, and $V \cap W$ deformation retracts to $S^{k-1}$. We have the Mayer--Vietoris exact sequence
\begin{align*}
\cdots \to H_k(V) \oplus H_k(W) &\to H_k\left(S_{R_{2} + \epsi} \cup \left(B_{R_{2}+\epsi}(0) \cap P^{k} \right)\right)\\ &\to H_{k-1}(V\cap W) \to H_{k-1}(V) \oplus H_{k-1}(W) \to \cdots.
\end{align*}
We see that $H_{k}(V), H_{k-1}(V), H_{k}(W), H_{k-1}(W)$ are all zero, so we have $$H_k\left(S_{R_{2} + \epsi} \cup \left(B_{R_{2}+\epsi}(0) \cap P^{k} \right)\right) \cong H_{k-1}(V\cap W) \cong H_{k-1}(S^{k-1}) \cong \mathbb{Z},$$ completing the proof of Theorem \ref{H^n betti bound} for $1 \leq k \leq n-2$.

Finally, we address the $k = n-1$ case. Here we can simply take our $S$ to be the union of tiles touching some sufficiently large closed ball $B_{R}$, minus the interior of the tile containing the center of the ball. Then the corresponding subgraph $\mathcal{S}$ is connected and satisfies the other conditions of $\ref{generalized theorem 8}$ by a similar argument as above, hence we can find sufficiently many copies of it. Via the same Mayer-Vietoris argument it suffices to show that each copy of $S$ has $n-1$th homology with rank at least $1$, but this follows from the fact that $S$ retracts to $S^{n-1}$, hence there is a surjection from the $n-1$th homology of $S$ to $H^{n-1}(S^{n-1}) \cong \mathbb{Z}$. 

\begin{rmk}
    Note that throughout this proof we have used the fact that balls and spheres in $\mathbb{R}^{n}$ centered at the origin and contained in the unit disk are exactly the balls and spheres centered at the origin in the disk model for hyperbolic space.
\end{rmk}

\section{Further generalizations} \label{S:further}
We have shown that in $\bb{H}^n$, the Eden model associated with a Voronoi tessellation satisfies Theorem \ref{H^n betti bound}.  In this section we generalize this theorem to a broader class of manifolds and actions on them: those with non-positive sectional curvature.

Let $X$ be a compact connected Riemannian manifold.  We endow the universal cover $\td{X}$ of $X$ with the Riemannian metric lifted from $X$, that is, the Riemannian metric that makes the covering map into a local isometry. The set of homeomorphisms (indeed, isometries) $\phi:\td{X}\to \td{X}$ such that the diagram
\[\xymatrix{
\td{X}\ar[rr]^\phi \ar[dr]_p & &\td{X}\ar[dl]^p\\
&X
}\]
commutes forms a group under composition. This group is called the \deffont{deck group} of $\td{X}$ and is isomorphic to the fundamental group $\pi_1(X)$.  This gives an action of $\pi_1(X)$ by isometries on $\td X$.

We want to generalize the Eden model to a tessellation of $\td{X}$ whose symmetries are given by the action of $\pi_1(X)$.  In the case that $X$ is the flat torus $T^n$ and $\td X=\bb R^n$, this reconstructs the usual Eden model; in the case that $X$ is a hyperbolic manifold, this gives some of the hyperbolic examples described above.  In other cases, we would like to divide $\td X$ into similar ``grid cubes'' or ``tiles''.

\subsection*{Voronoi cells and fundamental domains}
Often, one defines a fundamental domain in the universal cover $\td X$ to be a subset which contains exactly one point from each orbit of $\pi_1(X)$.  We modify this slightly: for us, a \deffont{(closed) fundamental domain} is a closed set $D$ whose translates under the action of $\pi_1(X)$ cover $\td X$ and such that for any $\phi \in \pi_1(X)$,
\[D \cap (\phi \cdot D) \subseteq \partial D.\]
The closure of a sufficiently nice fundamental domain in the usual sense fits this definition.  For example, in Euclidean space $\bb R^n$ thought of as the universal cover of a flat torus, a closed unit cube is a fundamental domain in our sense.  An example of a fundamental domain in the usual sense would be $[0,1)^n$.

A fundamental domain on $\td{X}$ can be obtained from the following construction. Pick a point $x\in\td{X}$, and consider its orbit $\ms{O}_x$ under the action of $\pi_1(X)$. An element in $\ms{O}_x$ is of the form $\phi\cdot x$ for some $\phi\in\pi_1(X)$. As before, the \deffont{Voronoi cell} of an element $x_\alpha \in \ms{O}_x$ is the set
\[V_\alpha=\bigl\{y\in \td{X}: \forall x_\beta\in \ms{O}_x\;d(y,x_\alpha)\leq d(y,x_\beta)\bigr\}.\]
Clearly, the Voronoi cells associated to $\ms O_x$ are translates of each other under the action of $\pi_1(X)$.  In addition:

\begin{lem}
For any $x \in \td X$, Voronoi cells associated with $\ms O_x$ are fundamental domains in $\td{X}$.
\end{lem}
This is proved in exactly the same way as Lemma \ref{Voronoi-tessellate}.

\begin{lem}\label{vcell bounded diam}
Voronoi cells in $\td{X}$ have bounded diameter: $\diam(V_\alpha)<\infty$. In particular, since all Voronoi cells are isometric, they must all have the same finite diameter.
\end{lem}

\begin{proof}
    Consider Voronoi cells $\{V_\alpha\}$ centered at points of the orbit $\ms{O}_x$.  Since $X$ is compact, it has finite diameter.  We will show that $\diam(V_\alpha) \leq 2\diam X$.
    
    Take a point $y\in\td{X}$.  Let $y'=p(y)$ and $x'=p(x)$, where $p:\td X \to X$ is the covering map. Connect $y'$ to $x'$ by a minimal geodesic $\gamma'$ in $X$. There is a unique lift $\gamma:[0,1]\to\td{X}$ with $\gamma(0)=y$ and $\gamma(1)=x_\alpha$ for some $x_\alpha\in\ms{O}_x$.  Therefore the distance from $y$ to the nearest point of $\ms O_x$ is bounded above by $d(x',y')$.  

    It follows that for any $y$ and $z$ located in the same Voronoi cell $V_\alpha$, there is a path from $y$ to $z$ through $x_\alpha$ of length at most $2\diam X$.
\end{proof}

\subsection*{The Eden model on $\td X$}
We generalize the Eden model to this setting the same way we did for hyperbolic tessellations: pick a Voronoi cell to start with; then at every time step, randomly choose a Voronoi cell adjacent to the infected region to infect next.  So that our model generalizes the Euclidean and hyperbolic ones, we say that two Voronoi cells are \deffont{adjacent} if their intersection is $(n-1)$-dimensional in the sense of Lebesgue covering dimension. We will use $A(t)$ to denote the Eden model at time $t$.

As before, we view each Voronoi cell $V_\alpha$ as the vertex $\alpha$ of a graph $T$, with an edge from $\alpha$ to $\beta$ in $T$ if and only if 
\[\dim(V_\alpha\cap V_\beta) \geq n-1.\]
\begin{lem}
  The graph $T$ is a Cayley graph for $\pi_1(X)$.
\end{lem}
In particular, this means that $T$ is infinite, locally finite, connected, and vertex-transitive.
\begin{proof}
  The proof is similar to that of Lemma \ref{lem:Cayley}.  Since $\pi_1(X)$ acts freely and transitively on $T$ by graph isomorphisms, by \ref{recognizeCayley} it suffices to prove that $T$ is locally finite and connected.  The proof of local finiteness is the same as for Lemma \ref{lem:Cayley}.

  It remains to show that $T$ is connected.  Let $\mathcal U$ be a nonempty, non-full subset of the vertices of $T$.  We will show that there is at least one edge between $\mathcal U$ and $\mathcal U^c$.  Denote the union of fundamental domains corresponding to vertices in a set $\mc A$ by $\langle \mc A \rangle$.  Then
  \[\td X=\langle \mc U \rangle \cup \langle \mc U^c \rangle\]
  and $\langle \mc U \rangle$ and $\langle \mc U^c \rangle$ both have nonempty interior.  Choose points $y \in \operatorname{int}\langle \mc U \rangle$ and $z \in \operatorname{int}\langle \mc U^c \rangle$ and let $V$ be a neighborhood homeomorphic to $\bb R^n$ of a simple curve between $y$ and $z$.  By \cite[Theorem 1.8.12]{Engelking},
  \[\dim(V \cap \langle \mc U \rangle \cap \langle \mc U^c \rangle) \geq n-1.\]
  On the other hand,
  \[\langle \mc U \rangle \cap \langle \mc U^c \rangle=\bigcup_{\alpha \in \mc U, \beta \in \mc U^c} V_\alpha \cap V_\beta.\]
  Since this is a countable union, and indeed a finite union if we restrict to any compact subset of $\td X$, by \cite[Corollary 50.3]{Munkres}, for some $\alpha$ and $\beta$, $V_\alpha \cap V_\beta$ is at least $(n-1)$-dimensional.
\end{proof}

Now suppose that $X$ has non-positive (but not necessarily constant) sectional curvature, $K\leq 0$.  Among other nice properties, this condition guarantees that every pair of points $x,y\in\td{X}$ is connected by a unique geodesic.  We will show the following generalization of Theorem \ref{H^n betti bound}:
\begin{thm} \label{NPC betti bound}
  Let $X$ be a compact connected manifold of non-positive sectional curvature.  Let $A(t)$ be the Eden model at time $t$ on a tessellation of $\td X$ by Voronoi cells of an orbit $\pi_1(X) \cdot x$.  Let $0\leq k\leq n-1$ and let $\beta_k(t)$ denote the $k$th Betti number of $A(t)$, i.e. $\beta_k(t)=\rank H_k(A(t))$.  Then there is a constant $C=C(X,k)>0$ such that
  \[\beta_k(t)\geq CF(|A(t)|)\]
  with high probability as $t\to\infty$, where $F:\bb N \to \bb N$ is the isoperimetric profile of $\pi_1(X)$.
\end{thm}
Here the isoperimetric profile of $\pi_1(X)$ is the isoperimetric profile of its Cayley graph for any finite generating set.  Since all such Cayley graphs are quasi-isometric, this is well-defined up to a multiplicative constant, as shown in Theorem \ref{quasiisometry invariance of isoperimetric profile}.  In the case that $X$ has not just non-positive but \emph{negative} sectional curvature, the fundamental group $\pi_1(X)$ is Gromov hyperbolic and therefore non-amenable (see Example \ref{IP examples}(c)).  It follows that we again have a linear lower bound.

The proof of this theorem is very similar to that of Theorem \ref{H^n betti bound}.  In fact, the only steps at which we use hyperbolicity are the construction of $P^k$ and $P^{n-k}$ and the proof of Lemma \ref{hyperbolic divergence}, which use the Poincar\'e disc model explicitly.  We replace this machinery by using the exponential map at a point $x \in \td X$.


%



\subsection*{The Exponential Map}
Let $M$ be a complete Riemannian manifold and let $p\in M$. Given a tangent vector $v\in T_p M$, by the Hopf--Rinow theorem, there is a unique geodesic $\gamma_v$ in $M$ with $\gamma_v(0)=p$ and $\gamma_v'(0)=v$. Then the \deffont{exponential map} $\exp_p:T_pM\to M$ is defined by
\[\exp_p(v)=\gamma_v(1).\]
If $M$ is non-positively curved and simply connected, then the Cartan--Hadamard theorem states that since there is a unique geodesic between $p$ and any other point, the exponential map is a diffeomorphism. Moreover, the derivative of the exponential map at the origin of the tangent space is the identity: $\dif(\exp_p)_0=\mathrm{id}$.

Alternatively, we may reparametrize all of our geodesics with unit speed.  Let $\eta_v$ be the unit speed parametrization of $\gamma_v$. Since the geodesic equation $\nabla_{\dot{\gamma_v}}\dot{\gamma_v}=0$ implies that $\gamma_v$ is a curve of constant speed, we see that $\gamma_v(1)=\eta_v(|v|)$. Therefore, we can equivalently think of the exponential map as the map $\exp_p:T_pM\to M$ defined by
\[\exp_p(v)=\eta_v(|v|).\]

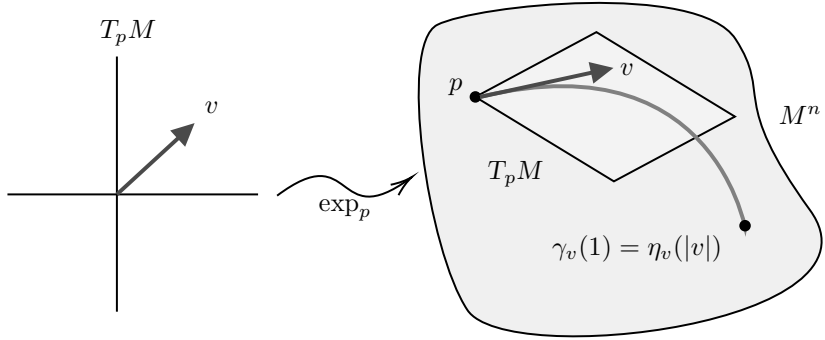
\begin{figure}
    \centering

\tikzset{every picture/.style={line width=0.75pt}} 

\begin{tikzpicture}[x=0.75pt,y=0.75pt,yscale=-1,xscale=1]

\draw [line width=0.75]    (94.4,150.96) -- (220.8,150.96) ;
\draw [line width=0.75]    (149.6,210.16) -- (149.6,81.36) ;
\draw  [fill={rgb, 255:red, 155; green, 155; blue, 155 }  ,fill opacity=0.15 ][line width=0.75]  (311.2,65.36) .. controls (331.2,55.36) and (442.8,43.6) .. (461.6,72.56) .. controls (480.4,101.52) and (457.6,101.36) .. (499.2,158.96) .. controls (540.8,216.56) and (346.4,239.36) .. (326.4,209.36) .. controls (306.4,179.36) and (291.2,75.36) .. (311.2,65.36) -- cycle ;
\draw  [fill={rgb, 255:red, 255; green, 255; blue, 255 }  ,fill opacity=0.21 ] (391.44,69.05) -- (461.29,111.68) -- (400.26,144.38) -- (330.4,101.76) -- cycle ;
\draw  [fill={rgb, 255:red, 0; green, 0; blue, 0 }  ,fill opacity=1 ] (332.8,101.76) .. controls (332.8,100.43) and (331.73,99.36) .. (330.4,99.36) .. controls (329.07,99.36) and (328,100.43) .. (328,101.76) .. controls (328,103.09) and (329.07,104.16) .. (330.4,104.16) .. controls (331.73,104.16) and (332.8,103.09) .. (332.8,101.76) -- cycle ;
\draw [color={rgb, 255:red, 74; green, 74; blue, 74 }  ,draw opacity=1 ][line width=1.5]    (149.6,150.96) -- (185.85,117.67) ;
\draw [shift={(188.8,114.96)}, rotate = 137.44] [fill={rgb, 255:red, 74; green, 74; blue, 74 }  ,fill opacity=1 ][line width=0.08]  [draw opacity=0] (11.61,-5.58) -- (0,0) -- (11.61,5.58) -- cycle    ;
\draw [color={rgb, 255:red, 128; green, 128; blue, 128 }  ,draw opacity=1 ][line width=1.5]    (332.8,101.76) .. controls (458.4,70.96) and (468,182.16) .. (466.4,166.16) ;
\draw  [fill={rgb, 255:red, 0; green, 0; blue, 0 }  ,fill opacity=1 ] (468.8,166.76) .. controls (468.8,165.43) and (467.73,164.36) .. (466.4,164.36) .. controls (465.07,164.36) and (464,165.43) .. (464,166.76) .. controls (464,168.09) and (465.07,169.16) .. (466.4,169.16) .. controls (467.73,169.16) and (468.8,168.09) .. (468.8,166.76) -- cycle ;
\draw    (230.4,151.6) .. controls (270,121.9) and (258.63,171.16) .. (297.22,143.04) ;
\draw [shift={(298.4,142.16)}, rotate = 143.13] [color={rgb, 255:red, 0; green, 0; blue, 0 }  ][line width=0.75]    (10.93,-3.29) .. controls (6.95,-1.4) and (3.31,-0.3) .. (0,0) .. controls (3.31,0.3) and (6.95,1.4) .. (10.93,3.29)   ;
\draw [color={rgb, 255:red, 74; green, 74; blue, 74 }  ,draw opacity=1 ][line width=1.5]    (332.8,101.76) -- (396.14,87.97) ;
\draw [shift={(400.05,87.12)}, rotate = 167.72] [fill={rgb, 255:red, 74; green, 74; blue, 74 }  ,fill opacity=1 ][line width=0.08]  [draw opacity=0] (11.61,-5.58) -- (0,0) -- (11.61,5.58) -- cycle    ;

\draw (139.6,61.6) node [anchor=north west][inner sep=0.75pt]    {$T_pM$};
\draw (334.88,132.6) node [anchor=north west][inner sep=0.75pt]  [xslant=-0.08]  {$T_{p} M$};
\draw (315.6,90.8) node [anchor=north west][inner sep=0.75pt]    {$p$};
\draw (192.4,102.8) node [anchor=north west][inner sep=0.75pt]    {$v$};
\draw (402,83.6) node [anchor=north west][inner sep=0.75pt]    {$v$};
\draw (367.6,170.2) node [anchor=north west][inner sep=0.75pt]    {$\gamma _{v}( 1) =\eta _{v}( |v|)$};
\draw (250,152) node [anchor=north west][inner sep=0.75pt]    {$\exp_{p}$};
\draw (482,104.6) node [anchor=north west][inner sep=0.75pt]    {$M^{n}$};

\end{tikzpicture}

\caption{Illustration of the map $\exp_p$.}

\end{figure}

Now take $M$ to be $\td X$ and the point $p$ to be the center $x$ of one of our Voronoi cells.  By construction, $\td X$ is simply connected, complete, and has non-positive sectional curvature.  Choose two orthogonal vector subspaces $H^k$ and $H^{n-k} \subseteq T_x \td X$ of the corresponding dimensions.  Define submanifolds $P^k=\exp_p(H^k)$ and $P^{n-k}=\exp_p(H^{n-k})$ in $\td{X}$.  Since the geodesic $\eta_v(t)$ only depends on the direction of the tangent vector $v$, the geodesic ray from the origin in $\bb{R}^n$ in the direction $v/|v|$ will be mapped onto the geodesic $\eta_v$ in $\td X$.  As in the hyperbolic construction, let $\cl{P^k}$ be the union of all Voronoi cells that intersect $P^k$ and $\cl{P^{n-k}}$ be the union of all Voronoi cells that intersect $P^{n-k}$ but not $P^k$.

We state the following result (which is a consequence of the Rauch comparison theorem) without proof:

\begin{lem}[Law of cosines]
Let $(M,g)$ be a complete simply connected Riemannian manifold with non-positive sectional curvature, and consider a geodesic triangle in $M$ whose side lengths are $a$, $b$, and $c$ with opposite angles $A$, $B$, and $C$ respectively. Then we have
\[a^2+b^2-2ab\cos(C)\leq c^2.\]
\end{lem}

This gives the following result:

\begin{cor} \label{NPC divergence}
    There exists $R_1>0$ such that for all vertices $\alpha$ of $T$ satisfying $d_T(\alpha,0)\geq R_1$, if the Voronoi cell $V_\alpha$ intersects $P^{n-k}$, then $V_\alpha \in \cl{P^{n-k}}$.  Here $0$ is the Voronoi cell centered at $x$.
\end{cor}

\begin{proof}
    Choose $R_0=\diam(V_\alpha)/\sqrt{2}$. Choose $y \in P^k$ and $z \in P^{n-k}$ such that $d(x,y)$ and $d(x,z)$ are both greater than $R_0$.  Then consider the geodesic triangle spanned by $x$, $y$, and $z$.  Since the exponential map takes lines through the origina to geodesics, the unique geodesic from $x$ to $y$ is the image of a straight line segment contained in $H^k$.  Similarly, the geodesic from $x$ to $z$ is the image of a straight line segment in $H^{n-k}$.  Moreover, since $\dif(\exp_p)_0=\text{id}$, it preserves the angle between geodesics through $x$.  Hence our triangle has a right angle at $x$, and by the law of cosines,
    \[d(x,y)^2\geq d(x,p)^2+d(y,p)^2>2R^2=\diam(V_\alpha)^2.\]
    Therefore, if $d(x,x_\alpha)>R_0$ and $V_\alpha$ intersects $P^{n-k}$, then $V_\alpha \in \cl{P^{n-k}}$.  Since $T \to \td X$ is a quasi-isometric embedding by the Milnor--Schwarz lemma, there is an $R_1$ such that if $d(0,\alpha)>R_1$, then $d(x,x_\alpha)>R_0$.  This completes the proof.
\end{proof}

\begin{figure}
    \centering

\tikzset{every picture/.style={line width=0.75pt}} 

\begin{tikzpicture}[x=0.75pt,y=0.75pt,yscale=-1,xscale=1]

\draw [line width=0.75]    (106.4,146.96) -- (232.8,146.96) ;
\draw [line width=0.75]    (162.6,206.16) -- (162.6,77.36) ;
\draw  [fill={rgb, 255:red, 155; green, 155; blue, 155 }  ,fill opacity=0.15 ][line width=0.75]  (323.2,61.36) .. controls (343.2,51.36) and (454.8,39.6) .. (473.6,68.56) .. controls (492.4,97.52) and (469.6,97.36) .. (511.2,154.96) .. controls (552.8,212.56) and (358.4,235.36) .. (338.4,205.36) .. controls (318.4,175.36) and (303.2,71.36) .. (323.2,61.36) -- cycle ;
\draw    (242.4,147.6) .. controls (282,117.9) and (270.63,167.16) .. (309.22,139.04) ;
\draw [shift={(310.4,138.16)}, rotate = 143.13] [color={rgb, 255:red, 0; green, 0; blue, 0 }  ][line width=0.75]    (10.93,-3.29) .. controls (6.95,-1.4) and (3.31,-0.3) .. (0,0) .. controls (3.31,0.3) and (6.95,1.4) .. (10.93,3.29)   ;
\draw  [fill={rgb, 255:red, 0; green, 0; blue, 0 }  ,fill opacity=1 ] (367.8,150.56) .. controls (367.8,149.23) and (366.73,148.16) .. (365.4,148.16) .. controls (364.07,148.16) and (363,149.23) .. (363,150.56) .. controls (363,151.89) and (364.07,152.96) .. (365.4,152.96) .. controls (366.73,152.96) and (367.8,151.89) .. (367.8,150.56) -- cycle ;
\draw  [fill={rgb, 255:red, 0; green, 0; blue, 0 }  ,fill opacity=1 ] (455.8,169.76) .. controls (455.8,168.43) and (454.73,167.36) .. (453.4,167.36) .. controls (452.07,167.36) and (451,168.43) .. (451,169.76) .. controls (451,171.09) and (452.07,172.16) .. (453.4,172.16) .. controls (454.73,172.16) and (455.8,171.09) .. (455.8,169.76) -- cycle ;
\draw  [fill={rgb, 255:red, 0; green, 0; blue, 0 }  ,fill opacity=1 ] (421.4,80.16) .. controls (421.4,78.83) and (420.33,77.76) .. (419,77.76) .. controls (417.67,77.76) and (416.6,78.83) .. (416.6,80.16) .. controls (416.6,81.49) and (417.67,82.56) .. (419,82.56) .. controls (420.33,82.56) and (421.4,81.49) .. (421.4,80.16) -- cycle ;
\draw [line width=1.5]    (108.8,183.12) -- (219.89,107.77) ;
\draw [shift={(223.2,105.52)}, rotate = 145.85] [fill={rgb, 255:red, 0; green, 0; blue, 0 }  ][line width=0.08]  [draw opacity=0] (11.61,-5.58) -- (0,0) -- (11.61,5.58) -- cycle    ;
\draw [line width=1.5]    (190.6,191.92) -- (127.94,92.9) ;
\draw [shift={(125.8,89.52)}, rotate = 57.67] [fill={rgb, 255:red, 0; green, 0; blue, 0 }  ][line width=0.08]  [draw opacity=0] (11.61,-5.58) -- (0,0) -- (11.61,5.58) -- cycle    ;
\draw [line width=1.5]    (365.4,150.56) .. controls (366.4,115.44) and (376.6,110.16) .. (416.6,80.16) ;
\draw [line width=1.5]    (365.4,150.56) .. controls (394.4,149.04) and (413.4,199.76) .. (453.4,169.76) ;
\draw  [dash pattern={on 4.5pt off 4.5pt}]  (419,80.16) .. controls (408.8,95.44) and (432.8,150.64) .. (455.8,169.76) ;
\draw [color={rgb, 255:red, 74; green, 74; blue, 74 }  ,draw opacity=1 ][line width=1.5]    (365.4,150.56) -- (365.58,110.64) ;
\draw [shift={(365.6,106.64)}, rotate = 90.26] [fill={rgb, 255:red, 74; green, 74; blue, 74 }  ,fill opacity=1 ][line width=0.08]  [draw opacity=0] (11.61,-5.58) -- (0,0) -- (11.61,5.58) -- cycle    ;
\draw [color={rgb, 255:red, 74; green, 74; blue, 74 }  ,draw opacity=1 ][line width=1.5]    (365.4,150.16) -- (403.2,150.23) ;
\draw [shift={(407.2,150.24)}, rotate = 180.11] [fill={rgb, 255:red, 74; green, 74; blue, 74 }  ,fill opacity=1 ][line width=0.08]  [draw opacity=0] (11.61,-5.58) -- (0,0) -- (11.61,5.58) -- cycle    ;

\draw (151.6,57.6) node [anchor=north west][inner sep=0.75pt]    {$T_pM$};
\draw (262,148) node [anchor=north west][inner sep=0.75pt]    {$\exp_{p}$};
\draw (498.8,91.8) node [anchor=north west][inner sep=0.75pt]    {$\tilde{X}$};
\draw (228.4,102.6) node [anchor=north west][inner sep=0.75pt]    {$H^{k}$};
\draw (81.2,84) node [anchor=north west][inner sep=0.75pt]    {$H^{n-k}$};
\draw (351.6,147.72) node [anchor=north west][inner sep=0.75pt]    {$p$};
\draw (424.4,67.72) node [anchor=north west][inner sep=0.75pt]    {$x$};
\draw (459.6,159.32) node [anchor=north west][inner sep=0.75pt]    {$y$};
\draw (386,178.92) node [anchor=north west][inner sep=0.75pt]    {$P^{n-k}$};
\draw (430,116.32) node [anchor=north west][inner sep=0.75pt]    {$d( x,y)$};
\draw (371.6,80.72) node [anchor=north west][inner sep=0.75pt]    {$P^{k}$};

\end{tikzpicture}
\caption{Illustration of the proof of Corollary \ref{NPC divergence}.}
\end{figure}
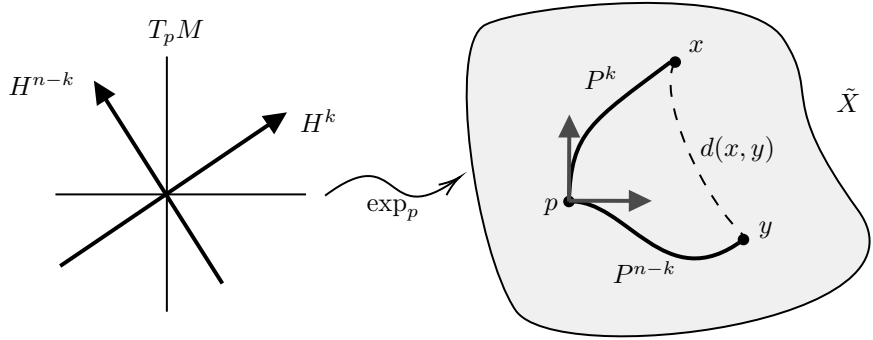

We have now constructed the first layer of the shell $\mc{B}_{R_1}(p)$ in Figure \ref{hyperbolic handle figure}, a handle centered at $x$.  The rest of the construction and proof works the same way it did in the hyperbolic case.

\begin{rmk}
    Note that the image of a ray from the origin in the tangent space under $\exp_{p}$ is the unique geodesic emanating from $p$ in that direction, and furthermore that when restricted to this ray $\exp_{p}$ is an isometry. So, we see that every ball (set of points distance $\leq R$ from $p$) or sphere (set of points exactly distance $R$ in the $p$) centered at $p$ in $\td X$ is the image under $\exp_{p}$ of a corresponding ball/sphere centered at the origin in the tangent space. Therefore, the union of a sphere in $\td X$ and the part of the hyperplane $\exp_{p}(H^{k})$ contained in the sphere is just the image under $\exp_{p}$ of the union of a sphere and a hyperplane portion in the tangent space. Since $\exp_{p}$ is a diffeomorphism, the two have the same homology, and so the proof of the hyperbolic case works here as well.
\end{rmk}

\subsection*{Discussion}

There is no obvious obstruction to generalizing the theorem further to manifolds which don't have non-positive curvature.  While our particular construction of handles with the required topological properties does not generalize, there doesn't seem to be a good reason why such handles cannot be constructed.  We therefore conjecture that Theorem \ref{NPC betti bound} holds for any manifold $X$.

Another notion worth discussing is the nontriviality of our results.  For example, if every Voronoi cell $V_\alpha$ had a $k$-handle of its own, then the inequality $\beta_k(E) \geq CF(|E|)$ would hold for \emph{every} union of Voronoi cells.  In that case, Theorem \ref{NPC betti bound} would not distinguish the Eden model from an arbitrary set of cells.

For hyperbolic tessellations, we know that this is never the case because Voronoi cells are convex.  In fact, since all intersections of Voronoi cells are convex and therefore contractible, a union of such cells is homotopy equivalent to its \v Cech complex.  In other words, the homology of such a union is determined by the lattice of intersections.

Nevertheless, it seems reasonable to suggest that for any manifold $X$, one can find fundamental domains in the universal cover that have this property:

\begin{conj}
  If $X$ is an aspherical closed manifold, then the action of $\pi_1(X)$ on $\td{X}$ admits a closed fundamental domain $D$ such that any nontrivial intersection of translates of $D$ under the $\pi_1(X)$-action is contractible.
\end{conj}

\section{The relationship between FPP and Eden time}\label{eden vs fpp}
Consider the exponential-time FPP model $A(t)$.  As previously discussed, it is equivalent to the Eden model up to time rescaling.  In this section, we show (informally speaking) that as the Eden time $|A(t)|$ approaches infinity, so does the FPP time $t$. Therefore, when we show a statement holds with high probability for the geometric and topological behavior of the FPP model, this is also the case for the Eden model.

\begin{lem} \label{lem:expected}
    The size $|A(t)|$ of the Eden Model on a regular tree of vertex degree $n$ at (FPP) time $t$ satisfies the recurrence relation
    \[\bb E(|A(t)|) = 2-e^{-t} + (n-1)\int_0^t e^{-s}\bb E(|A(t-s)|)\dif s.\]
\end{lem}
\begin{proof}
    We can view an $n$-ary tree as a root vertex with $n$ disjoint subtrees coming out of it. Each of these subtrees is an $n$-ary tree minus one of its subtrees.
    
    In order for a subtree to contribute to $|A(t)|$, the root vertex of that subtree must be infected within $t$ seconds. Therefore, the expected contribution from one such subtree is 
    \[\int_0^t e^{-s}\Bigl(1+\frac{n-1}{n}\bb E(|A(t-s)|-1)\Bigr)\dif{s}.\] Here, $e^{-s}$ is the probability density function for the passage time of the root of the subtree, and $1+\frac{n-1}{n}\bb E(|A(t-s)|-1)$ is the expected size of the subtree once the root of the subtree has been infected: it consists of the root plus $n-1$ additional subtrees.
    
    Since there are $n$ such subtrees, we conclude that 
    \begin{align*}
    \bb E(|A(t)|) &= 1 + n\int_0^t e^{-s}\Bigl(1+\frac{n-1}{n}\bb E(|A(t-s)|-1)\Bigr)\dif{s} \\
    &= 1 + n\int_0^t e^{-s}\Bigl(\frac{1}{n}+\frac{n-1}{n}\bb E(|A(t-s)|)\Bigr)\dif{s} \\
    &= 2-e^{-t} + (n-1)\int_0^t e^{-s}\bb E(|A(t-s)|)\dif{s}. \qedhere
    \end{align*}
\end{proof}

\begin{lem} \label{lem:tree-expectation}
    Let $A(t)$ be the Eden model at time $t$ growing on an $n$-ary regular tree. There exists a constant $C > 0$ such that 
    \[\bb E(|A(t)|) < Ce^{(n-1)t}\]
    for all $t$.
\end{lem}
\begin{proof}
    Let $C$ be a constant such that 
    \[\bb E(|A(t)|) < Ce^{(n-1)t}\]
    for all $t \in [0, n]$. Suppose, for contradiction, that there is a $t$ such that $\bb E(|A(t)|) \geq Ce^{(n-1)t}$. Since both sides are continuous in $t$, there exists some $t_{0}$ such that $\bb E(|A(t_{0})|) = Ce^{(n-1)t_{0}}$, and $\bb E(|A(t)|) < Ce^{(n-1)t}$ for all $t < t_{0}$. Note that we must have $t_{0} > n$.
    
    Then by Lemma \ref{lem:expected},
    \begin{align*}
        \bb E(|A(t_0)|) &= 2-e^{-t} + (n-1)\int_{0}^{t_{0}} e^{-s}\bb E(|A(t_{0}-s)|)\dif{s} \\
        &< 2 + (n-1)\int_{0}^{t_{0}} e^{-s}Ce^{(n-1)(t_{0}-s)}\dif{s}\\
        &= 2 + C(n-1)e^{(n-1)t_{0}}\int_{0}^{t_{0}} e^{-ns}\dif{s}\\
        &= 2 + C(n-1)e^{(n-1)t_{0}}\left(\frac{1}{n} - \frac{1}{n}e^{-nt_{0}}\right)\\
        &< 2 + C(n-1)e^{(n-1)t_{0}} \cdot \frac{1}{n}\\
        &< Ce^{(n-1)t_{0}}\\
        &= \bb E(|A(t_{0})|)
    \end{align*}
    where the inequality in the second-to-last line follows from the fact that 
    \[2 = \frac{1}{n} \cdot 2n < \frac{1}{n}e^{n} < \frac{1}{n}e^{t_{0}} \leq \frac{1}{n}e^{(n-1)t_{0}}.\]
    This is a contradiction, therefore such a $t_0$ cannot exist. Therefore, we have 
    \[\bb E(|A(t)|) < Ce^{(n-1)t}\]
    for all $t$.
\end{proof}

\begin{lem}\label{Eden Time FPP Time Equivalence for Trees}
    Let $A(t)$ be the Eden model on an $n$-ary regular tree $T$.  Let $t$ denote the FPP time and let $|A(t)|$ denote the size of the Eden model at (FPP) time $t$. Then there is a constant $D'>0$ such that
    \[|A(t)|\leq (D')^t\]
    with high probability as $t\to\infty$. 
\end{lem}

\begin{proof}
    The perimeter of a ball of radius $k$ in $T$ has size $n^k$.  To obtain an upper bound on the size of the subtree infected at time $t$, we compare our instance of the Eden model to one in which everything within a ball of radius $\lfloor\log_{n}t\rfloor$ immediately gets infected (everything within this ball has passage time 0) and all other passage times are unchanged.  Clearly, this modified model has a strictly larger infected region at time $t$.
    
    On the perimeter of the $\lfloor\log_{n}t\rfloor$-ball, there are $n^{\lfloor\log_{n}t\rfloor} < t$ disjoint subtrees coming out, each with its root on the perimeter. Each of these subtrees is a regular $n$-ary tree minus one of its $n$ primary subtrees (subtrees adjacent to the central vertex).  By Lemma \ref{lem:tree-expectation}, the expected size of an $n$-ary tree at time $t$ is bounded above by $CD^t$, where $D = e^{(n-1)} > 0$. Since each of the subtrees is a $n$-ary tree minus one of its main subtrees, its expected size at time $t$ is also certainly bounded above by $CD^{t}$. 
    
    Then as $t\to\infty$, by the law of large numbers, the total size of the $n^{\lfloor\log_{n}t\rfloor}$ subtrees will converge to $n^{\lfloor\log_{n}t\rfloor}E< tE < tD^t$, where $E$ is the expected size of each subtree. Thus as $t\to\infty$, with high probability we have
    \[|A(t)| \leq \text{Vol}(B_{\lfloor\log_{n}t\rfloor}) + 2n^{\lfloor\log_{n}t\rfloor}E <\text{Vol}(B_{\lfloor\log_{n}t\rfloor})+2CtD^t.\]
    Now 
    \[\text{Vol}(B_{\lfloor\log_{n}t\rfloor})=\sum_{k=0}^{\lfloor\log_{n}t\rfloor}n^k=\frac{n^{\lfloor\log_{n}t\rfloor}-1}{n-1}<\frac{t-1}{n-1}<t.\]
    Therefore, we get that
    \[|A(t)| < t+2CtD^t \leq t+(2CD)^t < 2^t+(2CD)^t \leq (2+2CD)^t.\]
    Setting $D' = 2+2CD$ concludes the proof.
\end{proof}

\begin{prop}
    Let $t$ denote the FPP time and let $|A(t)|$ denote the size of the Eden model at (FPP) time $t$ on an infinite, locally finite, vertex transitive graph $G$. Then there exists some constant $D>0$ such that
    \[|A(t)|\leq D^t\]
    with high probability as $t\to\infty$. 
\end{prop}

\begin{proof}
    Let $d=\deg G$. We use basically the same coupling trick we used in the proof of Lemma \ref{generalized lem 9}. But this time we couple $G$ with just one $(d-1)$-ary tree $T$, equipped with exponential i.i.d.\ passage times with mean $1$ on its nodes. The origin $o\in G$ of the infection is assigned passage time $0$ since it gets infected instantly. We consider a rooted $d$-ary tree $T$ with its root mapped to $O$ via the natural map $f:T\to G$. For each site $p\in G$, its passage time $\rho_p$ is defined to be $\rho_{\td{p}}$, where $\td{p}$ minimizes the passage time to just before $\td{p}$ among all vertices in $f^{-1}(p)$.
    
    As before, this defines a coupling between the Eden models on $T$ and $G$, since the only thing that depends on other vertices is \textit{which} i.i.d.\ exponentially distributed weight corresponds to a given site of $G$. By the same argument as in the proof of Lemma \ref{lem 3.11}, for every site $p\in G$ infected at time $t$, there exists a $\td{p}\in f^{-1}(p)$ such that the passage time from $\td{p}$ to the root of the tree $T$ is at most $t$.
    
    Therefore, the Eden model $A_G(t)$ on the graph $G$ at time $t$ is in bijection with a subset of the Eden model $A_T(t)$ on the tree at time $t$. Then Proposition \ref{Eden Time FPP Time Equivalence for Trees} implies that with high probability as $t\to\infty$, we have
    \[|A_G(t)|\leq |A_T(t)|\leq D^t\]
    for some constant $D>0$.
\end{proof}

\begin{cor}
  As a function of the Eden time $s$, the FPP time $t$ satisfies $t \geq \log_D s$ with high probability.
\end{cor}
\begin{proof}
    Suppose that $t \leq \log_D s$.  Then at FPP time $\log_D s$, we have
    \[|A(\log_D s)| \geq t=D^{\log_D s}.\]
    But the probability of this converges to $0$ as $\log_D s \to \infty$.  Therefore the probability that $t \leq \log_D s$ converges to $0$ as $s \to \infty$.
\end{proof}
This is the result we needed.

\appendix 

\section{Quasi-isometry invariance of the isoperimetric profile} \label{appendix}

In this section, we prove that (up to a natural equivalence relation) the isoperimetric profile is an quasi-isometry invariant of infinite, locally finite, vertex-transitive graphs. First, we give some definitions. 

\begin{defn}
Let $T$ be an infinite, locally finite, vertex transitive graph equipped with the graph metric. We define the \deffont{growth function} $N_{T} \colon \mathbb{N} \to \mathbb{N}$ of $G$ as
\[N_{T}(n) = |B_{n}(x)|.\]
In other words, the growth function returns the size of a (closed) ball of radius $n$ centered at an arbitrary point in $T$; the choice of center $x$ does not matter since $G$ is vertex-transitive.
\end{defn}

Lemma \ref{size of balls} shows that for any infinite, locally finite, vertex-transitive graph $T$ where each vertex has degree $d$,
\[N_{T}(n) \leq d^{n+1}.\] 




We can give some bounds on how much the size of a finite set of vertices of a graph can change under quasi-isometry.

\begin{lem}\label{quasi-isometry size changes}
Let $T$, $W$ be infinite, locally finite, vertex-transitive graphs equipped with the graph metric, and $f \colon T \to W$ be a $(C, K)$-quasi-isometry. Let $S \subset T$ be a finite subset. Then
\[\frac{1}{N_{T}(CK)}|S| \leq |f(S)| \leq |S|.\]
\end{lem}
\begin{proof}
The inequality $|f(S)| \leq |S|$ is true for any function. For the other inequality, we prove a bound on the size of the preimage $f^{-1}(y)$ of any point $y \in W$. 

Suppose $x, x' \in T$ satisfy $f(x) = f(x') = y$. Since $f$ is a quasi-isometry,
\[\frac{1}{C}d_{T}(x, x') -K \leq d_{W}(f(x), f(x')) = 0,\]
hence 
\[d_{T}(x, x') \leq CK.\]
It follows that for any point $x \in f^{-1}(y)$, all of $f^{-1}(y)$ is contained in the $CK$-ball around $x$.  Such a ball contains $N_{T}(CK)$ points, so we have 
\[|f^{-1}(y)| \leq N_{T}(CK).\]
This means that $f$ sends at most $N_{T}(CK)$ points in $S$ to the same point in $f(S)$, hence $\frac{1}{N_{T}(CK)}|S| \leq |f(S)|$, as desired.
\end{proof}

Now, we define an equivalence relation on functions from the naturals to the positive reals:
\begin{defn}
    Let $f, g$ be functions from $\mathbb{N}$ to $\mathbb{R}^{+}$. We write $g \lesssim f$ if there is a constant $a>0$ such that
    \[g(n) \leq af(n)\]
    for all $n \geq 1$. If we also have $f \lesssim g$, we write $f \sim g$. 
\end{defn}
One can check that $\sim$ is indeed an equivalence relation.

To show that the isoperimetric profile is a quasi-isometry invariant, we want to be able to take an isoperimetric set (one that has minimal boundary given its size) in one space and produce something close to isoperimetric in another quasi-isometric space. While taking a quasi-isometry can ``scatter'' the points of the isoperimetric set, the fact that a quasi-isometry is coarsely surjective allows us to overcome this obstacle: we can take the $K$-ball around the image of an isoperimetric set to obtain something close to being isoperimetric. The following theorem formalizes this:

\begin{thm}\label{quasiisometry invariance of isoperimetric profile}
Let $T, W$ be infinite, locally finite, vertex-transitive graphs equipped with the graph metric. If $T, W$ are quasi-isometric, then we have $F_{T} \sim F_{W}$. 
\end{thm}

\begin{proof}
To differentiate between balls in the two graphs, we will write $B_{T}(x, R)$ and $B_{W}(y, R)$ for the closed balls of radius $R$ around $x \in T$ and $y \in W$, respectively.

Fix some $n \geq 1$, and let $S_{T} \subseteq T$ be a subgraph of size $n$ satisfying $|\partial S| = F_{T}(n)$. Let $f \colon T \to W$ be a $(C, K)$-quasi-isometry. We define
\[S_{W} = \bigcup_{y \in f(S_{T})} B_{W}(y, K).\]

Consider the edge boundary $\partial S_{W}$. For each $y \in \partial S_{W}$, there exists an $x \in T$ such that $d_{W}(f(x), y) \leq K$ (due to $f$ being a quasi-isometry). For each $y \in \partial S_{W}$, we choose such an $x$, and let $X \subseteq T$ denote the collection of these $x$'s. 

It is possible that for distinct points $y, y' \in \partial S_{W}$, we choose the same corresponding $x \in X$. If this is the case, then both $y, y'$ are contained in the closed $K$-ball in $W$ centered at $f(x)$. Therefore, the maximum number of points $y \in \partial S_{W}$ which correspond to the same point $x \in X$ is bounded above by $N_{W}(K)$ (the size of a closed $K$-ball in $W$). It follows that 
\begin{equation} \label{ineq1}
|\partial S_{W}| \leq N_{W}(K)|X|.
\end{equation}

For each $x \in X$, $f(x)$ is within distance $K$ of a point in $\partial S_{W}$. Each point in $\partial S_{W}$ is within distance $1$ of a point $S_{W}$, and each point in $S_W$ is within distance $K$ of some point in $f(S_{T})$. Therefore, for each $x \in X$, there exists some $x' \in S_{T}$ such that 
\[d_{W}(f(x), f(x')) \leq 2K+1.\]
Since $f$ is a $(C, K)$-quasi-isometry, it follows that 
\[d_{T}(x, x') \leq C(3K+1).\]

Therefore, $X$ is contained in the closed $C(3K+1)$-neighborhood around $S_{T}$; in other words, we have 
\[X \subseteq \bigcup_{x \in S_T} B_{T}(x, C(3K+1)).\]
Denote this union by $B_T(S_T, C(3K+1))$.

On the other hand, we claim that $X \cap S_{T}$ is empty. Suppose, for a contradiction, that there exists a point $x \in X \cap S_{T}$. Since $x$ is in $X$, there exists some $y \in \partial S_{W}$ such that $d_{W}(f(x), y) \leq K$. However, since $x$ is also in $S_{T}$, and $S_{W}$ contains all the points in the closed $K$-neighborhood around $f(S_{T})$, it follows that $y$ is in $S_{W}$. But this contradicts the assumption that $y \in \partial S_{W}$.

Since $X \cap S_{T}$ is empty, we have 
\[X \subseteq B_T(S_T, C(3K+1)) \setminus S_{T}.\]

For each $j$, let $\partial^j S_T=\partial B_T(S_T,j-1)$.  Then $\partial^{j} S_{T}$ contains all the points whose distance to the nearest point in $S_{T}$ is exactly $j$, hence we have 
\[B_T(S_T, C(3K+1)) \setminus S_{T} = \bigcup_{j =1}^{C(3K+1)} \partial^{j} S_{T}.\]
Moreover, every element of $\partial^{j+1} S_T$ is adjacent to $\partial^j S_T$, which implies that $|\partial^{j+1} S_T| \leq \deg(T)|\partial^j S_T|$.  Therefore,
\begin{align*}
    \left|\bigcup_{j =1}^{C(3K+1)} \partial^{j} S_{T}\right| &
    \leq \sum_{j =1}^{C(3K+1)} \left|\partial^{j} S_{T}\right|\\
    &\leq \sum_{j =1}^{C(3K+1)} \deg(T)^{j-1} |\partial S_{T}|\\
    &\leq \deg(T)^{C(3K+1)} F_{T}(n).
\end{align*}
Since $X$ is a subset of $\bigcup_{j =1}^{C(3K+1)} \partial^{j} S_{T}$, it follows that
\begin{equation} \label{ineq2}
|X| \leq \deg(T)^{C(3K+1)}F_{T}(n)
\end{equation}
as well. Furthermore, since $f(S_{T}) \subseteq S_{W}$, it follows from Lemma \ref{quasi-isometry size changes} that
\[n=|S_{T}| \leq N_{T}(CK)|f(S_{T})| \leq N_{T}(CK)|S_{W}|.\]
Using the fact that $F_{W}$ is sublinear, it follows that
\[F_{W}(n) \leq F_W(N_T(CK)|S_W|) \leq N_T(CK)F_{W}(|S_W|) \leq N_T(CK)|\del S_W|.\]
Using \eqref{ineq1} and \eqref{ineq2}, we conclude that
\[F_{W}(n) \leq N_{T}(CK) N_{W}(K) \deg(T)^{C(3K+1)} F_T(n).\]
In particular, $F_{W} \lesssim F_{T}$.

Since $T$ and $W$ are quasi-isometric, there also exists a quasi-isometry $g \colon W \to T$. The proof above then shows that $F_{T} \lesssim F_{W}$, hence $F_{T} \sim F_{W}$.
\end{proof}


\subsection*{Acknowledgements}
This paper was written as part of a project for the UCSB REU.  We thank Katherine Merkl for many inspiring conversations, Maribel Bueno Cachadina for organizing the REU, and everyone else in the program for being extremely friendly and supportive. The authors were supported in part by NSF DMS-1850663 and a research fellowship from the College of Creative Studies at UCSB.

\subsection*{Conflict of interest statement}
On behalf of all authors, the corresponding author states that there is no conflict of interest.

\bibliographystyle{amsplain}
\bibliography{eden}

\end{document}